\documentclass[11pt,a4paper]{article}

\usepackage[T1]{fontenc} 
\usepackage[ocgcolorlinks,colorlinks=true,urlcolor=blue]{hyperref}
\usepackage[usenames,dvipsnames]{xcolor}
\usepackage[a4paper]{geometry}
\usepackage{amsmath,amssymb,amsthm,mathrsfs}
\usepackage{clrscode3e}
\usepackage{bbm}
\usepackage{times}
\usepackage{paralist}
\usepackage[numbers,sort&compress]{natbib}
\usepackage{graphicx}
\usepackage{framed}

\graphicspath{{Fig}{.}}

\newcommand{\red}{\color{red}}

\usepackage[normalem]{ulem}

\DeclareSymbolFont{bbold}{U}{bbold}{m}{n}
\DeclareSymbolFontAlphabet{\mathbbold}{bbold}

\newcommand{\pth}[1]{\ensuremath{\left(#1\right)}}
\newcommand{\sites}{\mathbf{X}}
\newcommand{\ppp}{\mathbf{\Phi}}
\newcommand{\DT}{\operatorname{DT}}

\newcommand{\reals}{\mathbb{R}}
\newcommand{\neighbours}{\mathbf{N}}

\newcommand{\cone}{\operatorname{Cone}}
\newcommand{\disc}{\operatorname{Disc}}
\newcommand{\elli}{\operatorname{Ell}}
\newcommand{\prob}{\mathbb{P}}
\newcommand{\expected}{\mathbb{E}}

\newcommand{\Po}{\operatorname{Po}}
\newcommand{\I}[1]{\mathbf{1}_{\{#1\}}}
\newcommand{\bX}{\mathbf X}
 
\newcommand{\cD}{\mathcal D} 
\newcommand{\cF}{\mathcal F}

\newcommand{\cQ}{\mathcal Q}

\newcommand{\cS}{\mathcal S}
\newcommand{\cA}{\mathcal A}
\newcommand{\cR}{\mathcal R}

\newcommand{\sC}{\mathscr C} 
\newcommand{\bbR}{\mathbb R}
\newcommand{\p}{\prob}
\newcommand{\e}{\expected}
\newcommand{\var}{\mathbb{V}}

\newcommand{\Wo}{\ensuremath W^\odot}
\newcommand{\expo}[1]{\exp\left(#1\right)}

\newtheorem{theorem}{Theorem}
\newtheorem{lemma}[theorem]{Lemma}

\newtheorem{proposition}[theorem]{Proposition}
\newtheorem{corollary}[theorem]{Corollary}

\theoremstyle{remark}
\newtheorem{remark}{Remark}

\author{
    Nicolas Broutin
        \thanks{Projet RAP, INRIA Paris - Rocquencourt}
    \and Olivier Devillers
        \thanks{Projet Geometrica, INRIA Sophia Antipolis - Méditerranée}
    \and Ross Hemsley
        \thanks{Projet Geometrica, INRIA Sophia Antipolis - Méditerranée}
}

\title{
Efficiently Navigating a Random\\ Delaunay Triangulation
\thanks{The work in this paper has been partially supported by ANR blanc 
        PRESAGE (ANR-11-BS02-003)}
}

\date{\today}

\begin{document}

\maketitle
\begin{abstract}
  Planar graph navigation is an important problem with significant implications
  to both point location in geometric data structures and routing in networks.
  However, whilst a number of algorithms and existence proofs have been
  proposed, very little analysis is available for the properties of the paths
  generated and the computational resources required to generate them under a
  random distribution hypothesis for the input. In this paper we analyse a
  new deterministic planar navigation algorithm with constant competitiveness
  which follows vertex adjacencies in the Delaunay triangulation. We call this
  strategy \emph{cone walk}. We prove that given $n$ uniform points in a smooth
  convex domain of unit area, and for any start point $z$ and query point $q$;
  cone walk applied to $z$ and $q$ will access at most $O(|zq|\sqrt{n} +\log^7
  n)$ sites with complexity $O(|zq|\sqrt{n} \log \log n + \log^7 n)$ with
  probability tending to 1 as $n$ goes to infinity. We additionally show that in
  this model, cone walk is $(\log ^{3+\xi} n)$-memoryless with high probability
  for any pair of start and query point in the domain, for any positive $\xi$.
  We take special care throughout to ensure our bounds are valid even when the
  query points are arbitrarily close to the border.
\end{abstract}




\section{Introduction}\label{sec:intro}
Given a planar embedding of a graph $G=(V,E)$, a source node $z\in V$ and a
destination point $q\in \reals^2$, we consider the \emph{planar graph
navigation} problem of finding a route in $G$ from $z$ to the nearest neighbour
of $q$ in $V$. In particular, we assume that any vertex $v\in V$ may access its
coordinates in $\reals^2$ with a constant time query. The importance of this
problem is two-fold. On the one hand, finding a short path between two nodes in
a network is currently a very active area of research in the context of routing
in networks~\cite{zg-wsnip-04,bb-sgwn-09, MiWoMi2009a}. On the other, the
problem of locating a face containing a point in a convex subdivision (point
location) is an important sub-routine in many algorithms manipulating geometric
data structures~\cite{l-scsi-77,dpt-wt-02,geometrica-7322i,Mucke:1996aa}. A
number of algorithms have been proposed within each of these fields, many of
which are in fact equivalent. It seems the majority of the literature in these
areas is concerned with the existence of algorithms which always succeed under
different types of constraints, such as the \emph{competitiveness} of the
algorithm, or the class of network. Apart from worst-case bounds, very little is
known concerning the properties of the path lengths and running times for these
algorithms under random distribution hypotheses for the input vertices. In this
paper we aim to bridge the gap between these two fields by giving and analysing
an algorithm that is provably efficient within both of these contexts when
the underlying graph is the Delaunay triangulation.

\subsection{Definitions}
In the following, we define the \emph{competitiveness} of an algorithm to be the
worst case ratio between the length of the path generated by the algorithm and
the Euclidean distance between the source and the destination. Thus
competitiveness may depend on the class of graphs one allows, but not the pair
of source and destination. Let $\neighbours_d(v)$ denote the set of neighbours
of $v$ within $d$ hops of $v$. We shall sometimes refer to the $d$'th
neighbourhood of a set $X$, to denote the set of all sites that can be accessed
from a site in $X$ with fewer than $d$ hops. We call an algorithm
\emph{$c$-memoryless} if at each step in the navigation, it only has access to
the destination $q$, the current vertex $v$ and $\neighbours_c(v)$. Some authors
use the term \emph{online} to refer to an algorithm that only has access to $q$,
the current vertex $v$, $\neighbours_1(v)$ and $O(1)$ words of memory which may
be used to store information about the history of the navigation. Finally, an
algorithm may be either deterministic or randomised. We define a randomised
algorithm to be an algorithm that has access to a random oracle at each step.

\subsection{Previous results}
\paragraph{Graph Navigation for Point location.}
The problem of point location is most often studied in the context of
triangulations and the algorithms are referred to as \emph{walking
algorithms}~\cite{dpt-wt-02}. A walking algorithm may work by following edges
or by following incidences between neighbouring faces, which is equivalent to a
navigation in the dual graph. There are three main algorithms that have
received attention in the literature: \emph{straight walk}, which is a walk
that visits all triangles crossed by the line segment $zq$; \emph{greedy vertex
walk}, which always chooses the vertex in $\neighbours_1(v)$ which is closest
to $q$ and \emph{visibility walk} which walks to an adjacent triangle if and
only if it shares the same half-space as $q$ relative to the shared edge. It is
known that these algorithms always terminate if the underlying triangulation is
Delaunay~\cite{dpt-wt-02}.

The aim is generally to analyse the expected number of steps that the algorithm
requires to reach the destination under a given distribution hypothesis. Such
an analysis has only been provided by \citet{dlm-etadp-04} who succeeded in
showing that straight walk reaches the destination after $O(\|zq\|\sqrt{n}\,)$
steps in expectation, for $n$ random points in the unit square (here and from
now on we shall use $\|\cdot \|$ to denote the Euclidean distance).\footnote{
    Zhu also provides a tentative $O\left(\sqrt{n\log n}\,\right)$ bound for
    visibility walk~\cite{z-lowrd-03}. This is a proof by induction for ``a
    random edge at distance $d$''. It considers the next edge in the walk and
    applies an induction hypothesis to try and bound the progress.
    Unfortunately, the new edge cannot be considered as random: each edge that
    is visited has been chosen by the algorithm and the edges do not all have
    the same probability to be chosen at each step in the walk. Restarting the
    walk from a given edge is not possible either (as done in
    \cite{DeMuZh1998a}), since the knowledge that an edge is a Delaunay edge
    influences the local point distribution.
}
The analysis in this case is facilitated since it is possible to compute the
probability that a triangle is part of the walk without looking at the other
vertices. Straight walk is online, but not memoryless since at every step the
algorithm must know the location of the source point, $z$. It is also rarely
used in practice since it is usually outperformed empirically by one of the
remaining two algorithms, \emph{visibility walk} or \emph{greedy vertex} walk,
which are both $1$-memoryless~\cite{dpt-wt-02}. The complex dependence between
the steps of the algorithm in these cases makes the analysis difficult, and it
remains an important open question to provide an analysis for either of these
two algorithms.

\paragraph{Graph Navigation for Routing.}
In the context of packet routing in a network, each vertex represents a node
which knows its approximate location and can communicate with a selected set of
neighbouring nodes. One example is in wireless networks where a node
communicates with all devices within its communication range. In such cases, it
is often convenient for the nodes to agree on a communication protocol such
that the graph of directly communicating nodes is planar, since this can make
routing more efficient. Triangulations have been used in this context due to
their ability to act as spanners (the length of shortest paths in the graph,
seen as curves in $\bbR^2$, do not exceed the Euclidean distance by more than a
constant factor), and methods exist to locally construct the full Delaunay
triangulation, given some conditions on the point
distribution~\cite{Wang:2007:EDL:1269960.1269963,Li:2003:LDT:2328757.2328820,
Kozma04geometricallyaware}.

Commonly referenced algorithms in this field are: \emph{greedy routing}, which
is the same algorithm as \emph{greedy vertex walk}, given in the context of
point location; \emph{compass routing} which is similar to greedy, except that
instead of choosing the point in $\neighbours_1(v)$ minimising the distance to
$q$, it chooses the point in $x\in \neighbours_1(v)$ minimising the angle
$\angle q,v,x$, and also \emph{face routing} which is a generalisation of
straight walk that can be applied to any planar graph. In this context, overall
computation time is usually considered less important than trying to construct
algorithms that find short paths in a given network topology under certain
memory constraints. We give a brief overview of results
relating to this work.

\citet{BoBrCaDe2002a} demonstrated that it is not possible to construct a
deterministic memoryless algorithm that finds a path with constant
competitiveness in an arbitrary triangulation. They also demonstrated by
counter example that neither greedy routing, nor compass routing is
$O(1)$-competitive on the Delaunay triangulation~\cite{BoMo2004a}.
\citet{DBLP:journals/tcs/BoseM04} went on to show that there does, however,
exist an online $c$-competitive algorithm that works on any graph satisfying a
property they refer to as the `diamond property', which is satisfied by Delaunay
triangulations. They show this by providing an algorithm which is essentially a
modified version of the straight walk. Bose and Morin also show that there is no
algorithm that is competitive for the Delaunay triangulation under the link
length (the link length is the number of edges visited by the
algorithm)~\cite{BoBrCaDe2002a}.

In terms of time analysis, it appears the only relevant results are those by
\citet{dlm-etadp-04}, where their results correspond with that of the straight
walk, and those by \citet{Chen2012178}, who show that no routing algorithm is
asymptotically better than a random walk when the underlying graph is an
arbitrary convex subdivision.

\paragraph{Navigation in the Plane.} We briefly remark that for the related
problem of navigation in the plane, several probabilistic results exist; for
example \cite{BoMa2011a} and \cite{Bordenave2008}. In this context, the input
is a set of vertices in the plane along with an oracle that can compute the
next step given the current step and the destination in $O(1)$ time. Although
the steps are also dependent in these cases, the case of Delaunay
triangulations we treat here is more delicate because of the geometry of the
region of dependence implied by the Delaunay property.

\subsection{Contributions}
In this paper we give a new deterministic planar graph navigation algorithm
which we call \emph{cone walk} that succeeds on any Delaunay triangulation and
produces a path which is $3.7$-competitive. We briefly underline the fact that
our algorithm has been designed for theoretical demonstration, and we do not
claim that it would be faster in a practical sense than, for example,
\emph{greedy routing} or \emph{face routing}. On the other hand, direct
comparisons would perhaps be unfair, since \emph{greedy routing} is not
$O(1)$-competitive on the Delaunay triangulation~\cite{BoMo2004a} whereas we
prove that cone walk is; and \emph{face routing} is not memoryless in any sense,
whilst cone walk is localised in the sense given by Theorem~\ref{thm:local}. In
the theorems that follow, we characterise the asymptotic properties of the cone
walk algorithm applied to a random input.

Let $\cD$ be a smooth convex domain of the plane with area $1$, and write
$\cD_n=\sqrt{n}\cD$ for its scaling to area $n$. For $x,y\in \cD$, let $\|xy\|$
denote the Euclidean distance between $x$ and $y$. Under the hypothesis that the
input is the Delaunay triangulation of $n$ points uniformly distributed in a
convex domain of unit area, we prove that, for any $\varepsilon>0$, our
algorithm is $O(\log^{1+\varepsilon} n)$-memoryless with probability tending to
one. In the case of cone walk, this is equivalent to bounding the number of
neighbourhoods that might be accessed during a step, which we deal with in the
following theorem.

\begin{theorem}\label{thm:local}
Let $\bX_n:=\{X_1,X_2,\dots, X_n\}$ be a collection of $n$ independent uniformly
random points in $\cD_n$. For $z\in \bX_n$ and $q\in \cD_n$, let $M(z,q)$ be the
maximum number of neighbourhoods needed to compute every step of the walk. Then,
for every $\epsilon>0$, $$\p\Big(\exists z \in \bX_n, q\in
\cD_n: M(z,q) > \log^{3+\epsilon} n \Big) \le \frac 1 n.$$ In particular, as
$n\to\infty$, $\e[\sup_{z\in \bX_n, q\in \cD_n} M(z,q)] = O(\log^{3+\epsilon}
n)$, for every $\epsilon>0$.

\end{theorem}
Also with probability close to one, we show that the path length, the number
of edges and the number of vertices accessed are $O(\,\|zq\|+\log^6
n\,)$ for any pair of points in the domain. We formalise these properties in the
following theorem.
\begin{theorem}\label{th:main-th}
Let $\bX_n:=\{X_1,X_2,\dots, X_n\}$ be a collection of $n$ independent uniformly
random points in $\cD_n$. Let $\Gamma(z,q)$ denote either the Euclidean length
of the path generated by the cone walk from $z\in \bX_n$ to $q\in \cD_n$, its
number of edges, or the number of vertices accessed by the algorithm when
generating it. Then there exist constants $C_{\Gamma,\cD}$ depending only on
$\Gamma$ and on the shape of $\cD$ such that, for all $n$ large enough,
\begin{align*}
\p\bigg(\exists z\in \bX_n,\, q\in \cD_n~:~\Gamma(z,q) 
>
C_{\Gamma,\cD} \cdot \|zq\| + 4(1+\sqrt{\|zq\|})\log^6 n\,\bigg)
&\le \frac 1 n.
\end{align*}
In particular, as $n\to\infty$,
$$
\e\bigg[\sup_{z\in \bX_n,\, q\in \cD_n} \Gamma(z,q)\bigg] = O( \sqrt{n}\,).
$$
\end{theorem}

Finally, we bound the computational complexity of the algorithm, $T(z,q)$.
\begin{theorem}\label{th:main-th}
Let $\bX_n:=\{X_1,X_2,\dots, X_n\}$ be a collection of $n$ independent uniformly
random points in $\cD_n$. Then in the RAM model of computation, there exists a
constant $C$ depending only on the shape of $\cD$ and the particular
implementation of the algorithm such that for all $n$ large enough,
\begin{align*}
\p\bigg(\exists z\in \bX_n,\, q\in \cD_n~:~T(z,q) 
>
C \cdot \|zq\|\log\log n + (1+\sqrt{\|zq\|})\log^6 n\,\bigg)
&\le \frac 1 n.
\end{align*}
In particular, as $n\to\infty$,
$$
\e\bigg[\sup_{z\in \bX_n,\, q\in \cD_n} T(z,q)\bigg] = O( \sqrt{n}\log\log n\,).
$$
\end{theorem}

\begin{remark}
The choice of the initial vertex is never discussed. However, previous results
show that choosing this point carefully can result in an expected asymptotic
speed up for any graph navigation algorithm~\cite{Mucke:1996aa}.
\end{remark}

\subsection{Layout of the paper}
In Section~\ref{sec:algo}, we give a precise
definition of the cone walk algorithm and prove some important geometric
properties. In Section~\ref{sec:analysis}, 
we begin the analysis for the cone walk algorithm applied to a homogeneous
planar Poisson process.
To avoid problems when the walk goes
close to the boundary, we provide an initial analysis which assumes that the
points are sampled from a disc with the query
point at its centre. 
This analysis is then extended to arbitrary query points
in the disc and also to other convex domains in Section~\ref{sec:relax_query}.
In Section~\ref{sec:area}, we prove estimates about an auxiliary line
arrangement which are crucial to proving the worst-case probabilistic bounds in
Theorems~\ref{thm:local} and~\ref{th:main-th}. Finally, we compare our findings
with computer simulations in Section~\ref{sec:simulations}.

\begin{figure}   [t]
   \begin{center}
       \includegraphics[scale=1]{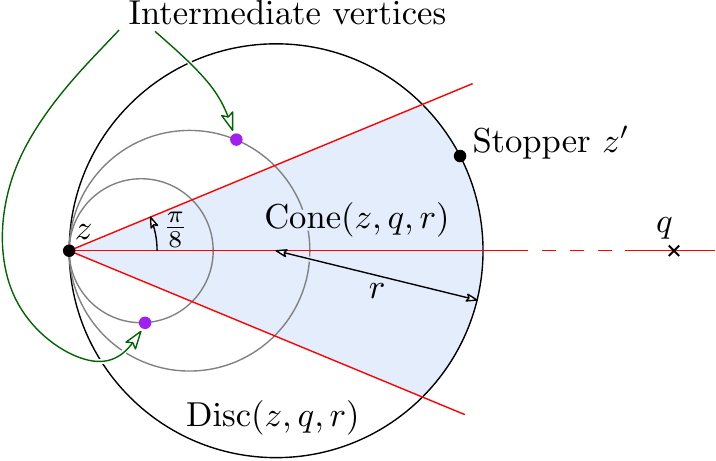}
   \end{center}
   \caption{Choosing the next vertex. \label{fig:next}}
\end{figure}

\section{Algorithm and geometric properties}\label{sec:algo}
We consider the finite set of sites in general position (so that no
three points of the domain are co-linear, and no four points are co-circular),
$\sites \subset \reals^2$
contained within a compact convex domain $\cD \subset \reals^2$. Let
$\DT(\sites)$ be the Delaunay triangulation of $\sites$, which is the graph in
which three sites $x,y,z\in \sites$ form a triangle if and only if the disc
with $x,y$ and $z$ on its boundary does not contain any site in $\sites$. Given
two points $z,q\in\reals^2$ and a number $r \in \reals$ we define
$\disc(z,q,r)$ to be the closed disc
whose diameter spans $z$ and the point at a distance $2r$ from $z$ on the ray
$zq$. Finally, we define $\cone(z,q,r)$ to be the sub-region of $\disc(z,q,r)$
contained within a closed cone of apex $z$, axis $zq$ and half angle
$\frac{\pi}{8}$ (see Figure~\ref{fig:next}).

\subsection{The cone walk algorithm}
Given a site $z\in\sites$ and a destination point $q\in\cD$, we define one
\emph{step} of the cone walk algorithm by growing the region $\cone(z,q,r)$
anchored at $z$ from $r=0$ until the first point $z'\in\sites$ is found such
that the region is non-empty. Once $z'$ has been determined, we refer to it as
the \emph{stopper}. We call the region $\cone(z,q,r)$ for the given $r$ a
\emph{search cone}, and we call the associated disc $\disc(z,q,r)$ the
\emph{search disc} (see Figure~\ref{fig:next}). The point $z'$ is then selected
as the anchor of a new search cone $\cone(z',q, \cdot)$ and the next step of
the walk begins. See Figure~\ref{fig:walk} for an example run of the algorithm.

To find the stopper using only neighbour incidences in the Delaunay
triangulation, we need only access vertices in a well-defined local
neighbourhood of the search disc. Define the points $\sites \cap \disc(z,q,r)
\setminus \{z,z'\}$ to be the \emph{intermediate vertices}. The algorithm finds
the stopper at each step by gradually growing a disc anchored at $z$ in the
direction of the destination, adding the neighbours of all vertices in $\sites$
intersected along the way. This is achieved in practice by maintaining a series
of candidate vertices initialised to the neighbours of $z$ and selecting amongst
them the vertex defining the smallest search disc at each iteration. Each time
we find a new vertex intersecting this disc, we check to see if it is contained
within $\cone(z,q,\infty)$. If it is, this point is the next stopper and this
step is finished. Otherwise the point must be an intermediate vertex and we add
its neighbours to the list of candidate vertices. This procedure works because
the intermediate vertex defining the next largest disc is always a neighbour of
one of the intermediate vertices that we have already visited during the current
step (see Lemma~\ref{lem:intermediate_steps}).

We terminate the algorithm when the destination $q$ is contained within the
current search disc for a given step. At this point we know that one of the
points contained within $\disc(z,q,r)$ is a Delaunay neighbour of $q$ in
$\DT(\sites \cup \{q\})$. We can further compute the triangle of $\DT(\sites)$
containing the query point $q$ (point location) or find the nearest neighbour
of $q$ in $\DT(\sites)$ by simulating the insertion of the point $q$ into
$\DT(\sites)$ and performing an exhaustive search on the neighbours of $q$ in
$\DT(\sites \cup \{q\})$.

We will sometimes distinguish between the \emph{visited} vertices, which we take
to be the set of all sites contained within the search discs for every step and
the \emph{accessed} vertices, which we define to be the set of all vertices
accessed by the cone-walk algorithm. Thus the accessed vertices are the visited
vertices along with their 1-hop neighbourhood.

The pseudo-code below gives a detailed algorithmic description of the
$\proc{Cone-Walk}$ algorithm. We take as input some $z\in\sites$, $q\in \cD$
and return a Delaunay neighbour of $q$ in $\DT(\sites \cup \{q\})$. Recalling
that $\neighbours_1(v)$ refers to the Delaunay neighbours of $v\in\DT(\sites)$
and additionally defining $\proc{Next-Vertex}(S, z, q)$ to be the procedure
that returns the vertex in $S$ with the smallest $r$ such that $\disc(z, q, r)$
touches a vertex in $S$ and $\proc{In-Cone}(z,q,y)$ to be \texttt{true} when
$y\in \cone(z,q,\infty)$.

\begin{footnotesize}
    \begin{codebox}
        \Procname{$\proc{Cone-Walk}(z,q)$}
        \li $Substeps    \gets \{z\}$
        \li $Candidates  \gets \neighbours_1(z)$
        \li \While \texttt{true} \Do 
            \label{code_main_loop}
        \li     $y \gets \proc{Next-Vertex}(Candidates \cup \{q\}, z,q)$
                \label{code_next_vertex}
        \li     \If $\proc{In-Cone}(z,q,y)$ \Do
        \li         \If $y=q$ \hspace*{5mm} \Do
                        \label{code_start_path}
        \li             \Comment{Destination reached.}
        \li             $\Return\; \proc{Next-Vertex}(Substeps,q,z)$
                    \End
        \li         \Comment{$y$ is a stopper}
        \li         $z\gets y$
        \li         $Substeps \gets \{z\}$
        \li         $Candidates \gets \neighbours_1(z)$
        \li     \Else 
        \li         \Comment{$y$ is an intermediate vertex.}
        \li         $Substeps\gets Substeps \cup \{y\}$
        \li         $Candidates \gets Candidates \cup
                    \neighbours_1(y)\setminus Substeps$
                \End
            \End
    \end{codebox}
\end{footnotesize}

\subsection{Path Generation} We note that the order in which the vertices are
discovered during the walk does not necessarily define a path in $\DT(\sites)$.
If we only wish to find a point of the triangulation that is close to the
destination (for example, in point location), this is not a problem. However,
in the case of routing, a path in the triangulation is required to provide a
route for data packets. To this end, we provide two options that we shall refer
to as $\proc{Simple-Path}$ and $\proc{Competitive-Path}$. $\proc{Simple-Path}$
is a simple heuristic that can quickly generate a path that is provably short
on average. We conjecture that $\proc{Simple-Path}$ is indeed competitive,
however we were unable to prove this.
$\proc{Competitive-Path}$ is slightly more complex from an implementation point
of view, however we show that for any possible input the algorithm will always
generate a path of constant competitiveness whilst still maintaining the same
asymptotic behaviour under the point distribution hypotheses explored in
Section~\ref{sec:analysis}.

\paragraph{$\proc{Simple-Path}$}
A simple way to generate a valid path is to keep a predecessor table for each
vertex. We start with an empty table at the beginning of each step, and then
every time we access a new vertex, we store it in the table along with the
vertex that we accessed it from. To trace a path back, we simply follow the
predecessors.

\paragraph{$\proc{Competitive-Path}$}
Let $Z_i$ for $i>0$, be the $i$th stopper in the walk thus, $Z_{i}$ is the
stopper found at step $i$) and $Z_0 := z$. For a path to be competitive, it
should at least be
\emph{locally} competitive: for each step, there should be a bound on the
length of the path generated between $Z_i$ and $Z_{i+1}$, which does not depend
on the points in the search disc. To construct a path verifying this property,
we use the fact that the \emph{stretch factor} of the Delaunay triangulation is
bounded above by a constant, $\lambda$. This means that for any two sites $x,y$,
there exists a path from $x$ to $y$ in the Delaunay triangulation for which the
sum of the lengths of the edges is at most $\lambda \|xy\|$. Currently the
literature gives us that the stretch factor is in $[1.5932,
1.998]$~\cite{Xia2011a,DBLP:conf/cccg/XiaZ11}. Clearly this implies that there
exists a path between $Z_i$ and $Z_{i+1}$ with total length at most
$\lambda
\|Z_iZ_{i+1}\|$, and this path cannot exit the ellipse
$\elli(Z_i,Z_{i+1}):=\{x\in\sites : \|xZ_i\|+ \|xZ_{i+1} \| <
\lambda\|Z_iZ_{i+1}\| \}$. We use Dijkstra's algorithm to find the shortest
path between $Z_i$ and $Z_{i+1}$ which uses only vertices within
$\elli(Z_i,Z_{i+1})$. The resulting path implicitly has stretch bounded by
$\lambda$. We show in Lemma~\ref{lem:competitiveness} that this algorithm
results in a bound for the competitiveness for the full path.

\begin{lemma}
\label{lem:competitiveness} $\proc{Cone-Walk}$ is $3.7$-competitive when the
{\sc Competitive-Path} algorithm is used to generate the path in $\DT(\sites)$
\end{lemma}
\begin{proof}
Let $Z_i$, $Z_{i+1}$ be the stoppers of two consecutive steps defined by
the algorithm. The stretch factor bound guarantees that the path generated
between $Z_i$ and $Z_{i+1}$ has length bounded by $\lambda \|Z_i Z_{i+1}\|$,
meaning that the longest path can have stretch at most $\lambda
\sum_{i=0}^{\tau-1}\|Z_i Z_{i+1}\|/\|zq\|$ where $\tau$ is the number of steps
in the walk. We bound this sum by observing that $\|Z_i Z_{i+1}\| \leq 2 \cos
\tfrac{\pi}{8} \cdot(\|Z_i q\| - \|Z_{i+1}q\|) $, which follows from
Figure~\ref{fig:dist-progress}. Finally, no path defined by the algorithm can be
longer than
\begin{align*}
      \lambda \sum_{i=0}^{\tau-1}\|Z_i Z_{i+1}\|
      \leq 2 \lambda \cos \tfrac{\pi}{8} 
          \sum_{i=0}^{\tau-1}\big(\|Z_i q\| - \|Z_{i+1}q\|\big)
      \leq 2 \lambda \cos \tfrac{\pi}{8} \cdot \|z q\|
\end{align*}
Thus the path is $c$-competitive for 
$c := 2 \lambda \cos \tfrac{\pi}{8}  \le 4 \cos \tfrac{\pi}{8} \le  3.7$.
\end{proof}

\subsection{Complexity}
\label{sec:complexity}
In this section we give deterministic bounds on the number of operations
required to compute $\proc{Cone-Walk}(z,q)$ within the RAM model of computation.
In this model, accessing, comparing and performing arithmetic on points is
treated as atomic. We will use these deterministic bounds to extract
probabilistic bounds under certain distribution assumptions in
Section~\ref{sec:algorithmic_complexity}. For now, we focus on a single step of
the walk starting from $y\in\sites$, and resulting in a disc with radius $r$.
Let $k$ be the number of points intersecting the disc $\disc(y,q,r)$ and $m$ be
the number of edges in $\DT(\sites)$ intersecting $\partial\disc(y,q,r)$
(where we use the notation $\partial A$ to denote the boundary of $A$) .

We note that every intermediate vertex will add its neighbours to the list of
$Candidates$ when visited. Each of these insertions can be associated with a
single edge of $\DT(X)$ intersecting $\disc(y,q,r)$ (with multiplicity two
for each `internal' edge, since they are accessed from both sides). By the
Euler relation, the total number of such insertions for one step is thus at most
$3(m+2k)$.
In addition, we observe that when moving from one intermediate vertex
to the next, a search in the list of $Candidates$ is required. A simple linear
search requires $O(m+k)$ operations for each intermediate vertex. Combining this
with the above, we achieve a bound of $O(k(m+k))$ operations for one step. This
bound may be improved by replacing $Candidates$ with a priority queue keyed on
the associated search-disc radius of each candidate, which yields a simple
improvement to $O(k\log(m+k))$.

For the path generation algorithms, we observe that $\proc{Simple-Path}$ only
requires a constant amount of processing per vertex accessed to generate the
predecessor table and $O(k)$ time to output the path at the end of each step,
so the asymptotic running time is not affected by its inclusion.
$\proc{Competitive-Path}$ is slightly more complicated since it accesses all
points within an ellipse enclosing each search disc. Let $k'$ be the number of
points in this ellipse along with their neighbours. The path is found by
applying Dijkstra's algorithm to $k'$ points, applying Euler's relation gives
us an updated bound for a single step of $O(k'\log k')$.
\begin{figure} [t]
       \begin{center}
        \includegraphics[scale=0.8]{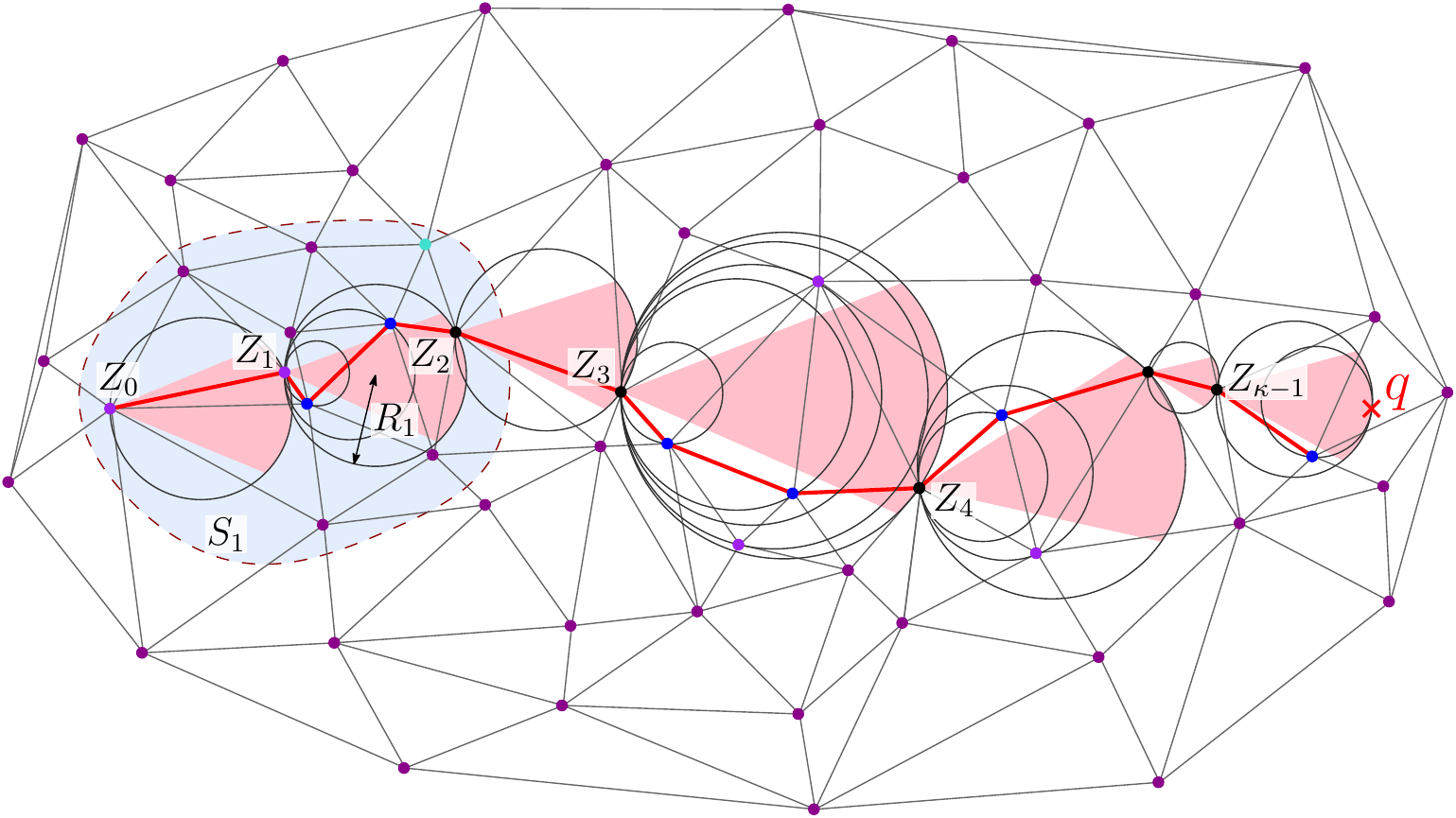}
    \end{center}
    \caption{
       An example of cone  walk. The points $(Z_i)_{i>0}$ are the
       \emph{stoppers} (or sometimes, the \emph{steps}) and the points within
       each circle are the \emph{intermediate vertices}.
       The shaded conic regions are the \emph{cones}, and the set of outer
       circles for each of the steps in the diagram is referred to as the
       \emph{discs}.
       \label{fig:walk}
    }
\end{figure}

\subsection{Geometric properties}\label{sec:geometry}
We now prove a series of geometric lemmata giving properties of steps in
the walk. We begin with a small lemma that will guarantee that we never get
`stuck' when performing a search for the next step, thus demonstrating
correctness of the algorithm. The following two `overlapping' lemmata allow us to
establish which regions may be considered independent in a probabilistic sense
and will be important in Section~\ref{sec:analysis}. Finally we provide a
`stability' result, which will help us to bound the region in which a
destination point may be moved without changing the sequence of steps taken by
the algorithm. This will be important when we enumerate the number of different
walks possible for a given set of input points.

\subsubsection{Finding a Delaunay path within the discs}
\begin{lemma}[Path finding lemma]\label{lem:intermediate_steps}
Let $q\in \cD$, $z\in \sites$ and $y' \in \sites$ with associated disc
$\disc(z,q,r')$. Suppose there exists an $r>0$ such that $\left(\disc(z,q,r')
\setminus \disc(z,q,r)\right) \cap \sites= \{y'\}$. Then there exists a point in
$\disc(z,q,r)$ that is a Delaunay neighbour of $y'$.
\end{lemma}
\begin{figure}
    [t]
    \begin{center}
        \includegraphics[scale=1]{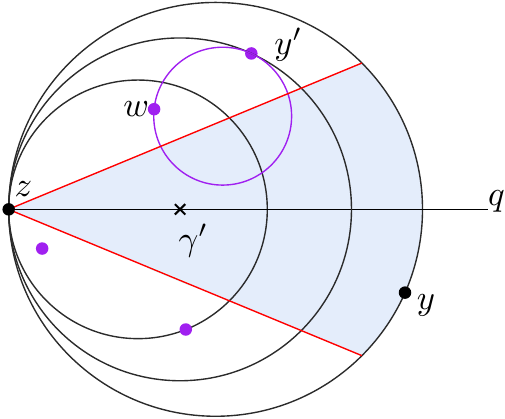}
    \end{center}
    \caption{
        We observe that $y'$ has always a Delaunay neighbour 
        in $\disc(z,q,r)$, where $r$ is the
        radius ensuring $y\in \partial\! \disc(z,q,r)$.
        \label{fig:algorithm_step}
    }
\end{figure}
\begin{proof}
    Let $\gamma'$ be the centre of $\disc(z,q,r')$. We grow 
    $\disc(y', \gamma',\rho) \subset \disc(z,q,r')$ 
    until we hit a point $w$ in $\sites$. The
    point $w$ is always contained within $\disc(z,q,r)$ because $z$ is on the
    border of $\disc(z,q,r)$. Since the interior of the
    $\disc(y',\gamma',\rho)$ is empty, $w$ is a Delaunay neighbour of $y'$. See
    Figure~\ref{fig:algorithm_step}.
\end{proof}

\begin{corollary}
    Let $q\in \cD$, $z\in \sites$ with $y\in \sites$ its 
    associated stopper satisfying
    $y\in\partial\!\cone(z,q,r)$.  Then there is a path of edges
    of \, $\DT(\sites)$ between $z$ and $y$  contained within $\disc(z,q,r)$.
\end{corollary}

\subsubsection{Independence of the search cones}
When growing a new search cone, it is important to observe that it does not
overlap any of the previous search cones, except at the very end of the walk.
This is formalised by the following lemma.

\begin{figure}[t]
  \begin{center}
       \includegraphics[scale=1]{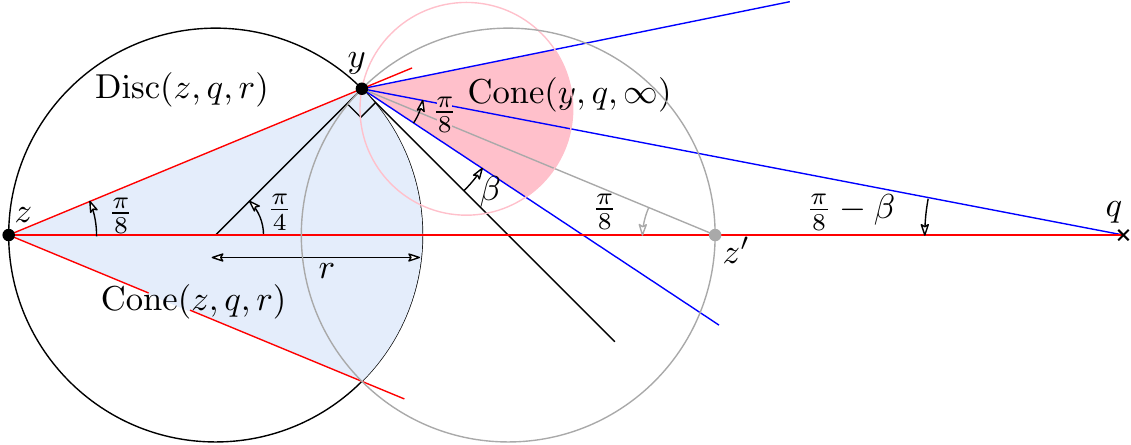}
   \end{center}
   \caption{
       For the proof of Lemma~\ref{lem:independent-cones}.
       \label{fig:independent-cones}
   }
\end{figure}

\begin{lemma}[Non-overlapping lemma]\label{lem:independent-cones}     
    Let $z$ and $y$ be two points of\, $\sites$ and $r>0$ such that
    $\cone(z,q,r)$ has $y$ on its boundary.
    If $\| zq \| > (2+\sqrt{2}\,) r$
    { then $\disc(z,q,r)$ does not intersect }
    the search cone $\cone(y,q, \infty)$ issued from $y$ 
    nor any other search cone for any subsequent step of the walk.
\end{lemma}
\begin{proof}
    Assume without loss of generality that $y$ lies to the left of line $zq$
    and consider the construction given in 
    Figure~\ref{fig:independent-cones}.
    Let $\beta$ denote the angle between the tangent to $\disc(z,q,r)$ at $y$
    and the ray bordering $\cone(y,q,\infty)$.
    $\cone(y,q,\infty)$ and  $\disc(z,q,r)$ do not intersect provided that
    $\beta\geq 0$.
    Placing $y$ at the corner of $\cone(z,q,r)$ maximizes $\beta$,
    in which case we have $\beta>0$ if
    and only if $q$ is to the right of $z'$, the point symmetrical to $z$
    with respect to the line through $y$ perpendicular to $zq$.
    Elementary computations then yield the result.
    Since the whole sequence of search cones following
    the one issued from $y$ remains in { $\cone(y,q,\infty)$, }    
    $\disc(z,q,r)$ does not intersect any of these
    search cones, and the result follows.
\end{proof}

\subsubsection{Independence of the search discs}
When growing the search disc region, the new search disc may overlap previous
search discs but only in their cone parts. This is formalised by the following
lemma:

\begin{lemma}[Overlapping lemma]\label{lem:independent-discs}     
    Let $z$ and $y$ be two  points such that $\cone(z,q,r)$ has $y$ on
    its boundary.
    Then if the search disc $\disc(y,q,\rho)$ issued from $y$ does not
    contain $q$, it
    does not intersect  $\disc(z,q,r)\setminus \cone(z,q,r)$.
\end{lemma}
\begin{figure}[t]
 \begin{center}
      \includegraphics[scale=1]{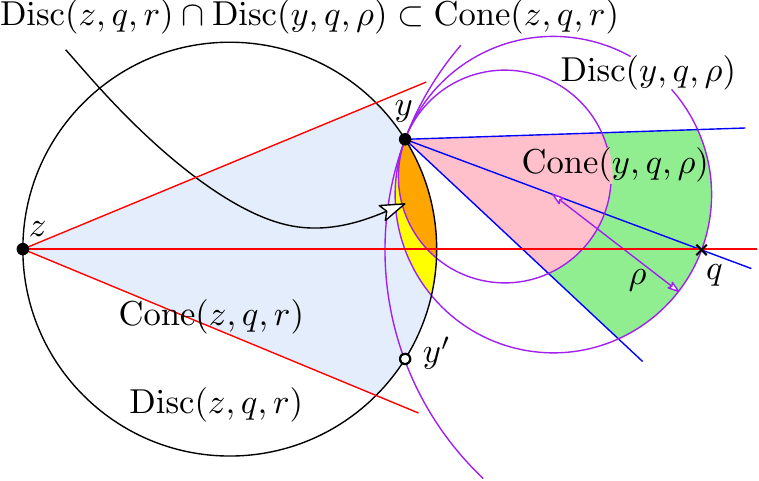}
  \end{center}
  \caption{
      For the proof of Lemma~\ref{lem:independent-discs}. 
      \label{fig:independent-discs}}
\end{figure}
\begin{proof}
    By symmetry we observe that $\disc(y,q,\rho)$ only intersects the point
    $y'$, the point $y$ reflected through the line $zq$, when the centre of
    $\disc(y,q,\rho)$ coincides with $q$ 
    Figure~\ref{fig:independent-discs}). Since the algorithm terminates as soon
    as the current search disc touches $q$, $q$ is never contained within
    $\disc(z,q, \rho)$ and thus this can never happen.
\end{proof}

\subsubsection{Stability of the walk}\label{sec:stability}
In the following lemma we are interested in the stability of the sequence of
steps to reach $q$.

\begin{lemma}[Invariance lemma]\label{lem:invariance} 
For an $n$-set $\bX\subseteq \bbR^2$, there exists an arrangement of half-lines
$\Xi=\Xi(\bX)$ such that the associated subdivision of the plane has fewer than
$2 n^4$ cells, and such that the sequence of steps used by the cone walk
algorithm from any vertex of $\bX$ does not change when the aim $q$ moves in a
connected component of $\bbR^2\setminus \Xi$.
\end{lemma}

\begin{figure}[t]
 \begin{center}
      \includegraphics[scale=0.8]{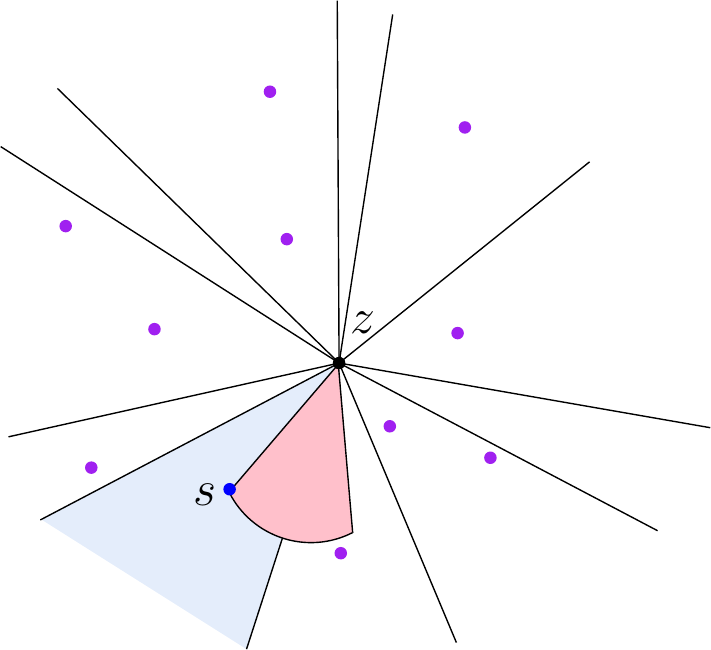}
  \end{center}
  \caption{
      For the proof of Lemma~\ref{lem:invariance}. For a given step $z$,
      moving the destination in the shaded sector will always result in the same
      stopper, $s$, being chosen for the next step.
      \label{fig:stoppers}
  }
\end{figure}
\begin{proof}
Take a point $z\in\sites$ and consider $\mathcal{S}_z$, the set of all possible
stoppers defined by $\cone(z,q,r)$ for some $q\in\cD$ and $r>0$. Each
$s\in\mathcal{S}_z$ defines a unique sector about $z$ such that moving a point
in the given sector does not change the stopper  (see
Figure~\ref{fig:stoppers}). We then create an arrangement by adding a ray on
the border of every sector for each point $z\in\sites$. The resulting
arrangement has the property that moving the destination point $q$ within one
of the cells of the arrangement does not change the stopper of any step for any
possible walk. Clearly, $|\mathcal{S}_z| \leq n-1$ for all $z\in\sites$, and
each sector is bounded by at most two rays, thus there are at most $2 n(n-1)$
rays in the arrangement. Since an arrangement of $m$ lines has at most
$\frac{m^2+m+2}{2}$ cells the result follows { \cite[see, e.g., ][p.
127]{Matousek2002a}}.
\end{proof}


\section{ Cone walk on Poisson Delaunay in a disc}\label{sec:analysis}

Our aim in this section is to prove the main elements towards
Theorem~\ref{th:main-th}, which we go on to complete in
Section~\ref{sec:relax_query}.
Our ultimate goal is to prove bounds on the behaviour of the cone walk for the
\emph{worst} possible pair of starting point and query when the input sites are
generated by a homogeneous Poisson process in a compact convex domain. Achieving
this requires first strong bounds on the probability that the walk behaves badly
for a fixed start point and query. One then proves the worst-case bounds by
showing that to control every possible run of the algorithm, it suffices to
bound the behaviour of the walk for enough pairs of starting points and query;
this relies crucially on the arrangement of Lemma~\ref{lem:invariance}. The tail
bounds required in the second stage of the proof may not be obtained from Markov
or Chebyshev's inequalities together with mean or variance estimates only, and
we thus need to resort to stronger tools.

Our techniques rely on \emph{concentration inequalities}
\cite{DuPa2009,BoLuMa2012a,McDiarmid1998,Janson2004b}. Most of the bounds we
obtain (for the number of steps $\kappa$ and the number of visited sites)
follow from a representation as a sum of random variables in which the
increments can be made independent by a simple and natural conditioning. The
bounds on the complexity of the algorithm \proc{Cone-Walk} are slightly
trickier to derive because there is no way to make the increments independent.

\medskip
For the sake of presentation, we introduce two simplifications which we remove
in Section~\ref{sec:relax_query}. First, we start by studying the walk in the
disc $\cD_n$ of area $n$ where the query is at the centre. These choices for
$\cD_n$ and $q$ ensure that for any $z\in\cD_n$ and any $r\le \sqrt{n/\pi},$ we
have $\disc(z,q,r)\subset\cD_n$. Note that since the distance to the aim is
decreasing, the disc is precisely the \emph{effective} domain where the walk
from $z$ and aiming at $q$ takes place.

Then, we introduce independence between the different regions of the domain by
replacing the collection of independent points $\bX_n$ by a (homogeneous)
Poisson point process $\ppp$ and consider $\DT(\ppp)$. Recall that a Poisson
point process of intensity $1$ is a random collection of points $\ppp \subset
\cD_n$ such that with probability one, all the points are distinct, for any two
Borel sets $R,S\subseteq \cD_n$, the number of points $|\ppp \cap R|$ is
distributed like a Poisson random variable whose mean is the area $\cA(R)$ of
$R$, and if $R\cap S=\varnothing$ then $|\ppp\cap R|$ and $|\ppp \cap S|$ are
independent.

On many occasions, it is convenient to consider $\ppp$ conditioned to have a
point located at $z\in \cD$ and we let $\ppp_z$ be the corresponding random
point set. Classical results on Poisson point processes ensure that $\ppp_z
\setminus \{z\}$ is distributed like $\ppp$, so that one can take $\ppp_z=\ppp
\cup \{z\}$, for $\ppp$ independent of $z$ \cite[see, e.g.,][Section
1.4]{bb-sgwn-09}.

\subsection{Preliminaries}
We establish the following notation (see Figure~\ref{fig:walk}). Let ${\mathbf
Z}=(Z_i, i>0)$ denote the sequence of stoppers visited during the walk with
$Z_0:=z$. Let $L_i=\|Z_iq\|$ denote the distance to the destination $q$. The
distance $L_i$ is strictly decreasing and the point set $\ppp$ is almost surely
finite, thus ensuring that the walk stops after a finite number of steps
$\kappa$, at which point we have $Z_\kappa=q$. For $x>0$, we also let
$\kappa(x)$ be the number of steps required to reach a point within distance $x$
of the query. Therefore $i< \kappa(x)$ if and only if $L_{i}> x$. The important
parameters needed to track the location and progress of the walk are the radius
$R_i$ such that $Z_{i+1}\in\partial\!\cone(Z_i,q,R_i)$, and the angle $\alpha_i$
between $Z_iq$ and $Z_iZ_{i+1}$. $\disc(Z_i,q,R_i)$ may contain several points
of $\ppp$, let $\tau_i$ denote
$|\disc(Z_i,q,R_i)\setminus\{Z_i,Z_{i+1}\}\cap{\ppp}|$ the number of such points
and $N_i$ the number of these points along with their Delaunay neighbours.

In order to compute the walk efficiently, the algorithm presented gathers a lot
of information. In particular, we access all of the points in
$\disc(Z_i,q,R_i)$ and their neighbours. For the analysis, we want to keep the
landscape as concise as possible, and so we define a filtration which only
contains the necessary information for the walk to be a measurable process. Let
$\cF_i$ denote the information consisting of (the $\sigma$-algebra generated
by) the locations of the points of $\ppp$ contained in $\cup_{j=0}^{i}
\disc(Z_j,q,R_j)$. Finally, we shall write $\omega_n$ to denote a sequence
satisfying $\omega_n\ge \log n$.

We often need to condition on the size of the largest empty ball within the
process $\ppp_n$. This is dealt with in the following lemma. 
\begin{lemma} \label{lem:largest_empty_disc}
    Let $b(\mathbf{x}, r)$ denote the closed ball of radius $r$ centred at
    $\mathbf{x}$. Then $\forall c >0$, $\xi>0$,
    $$
      \p\big(
      \exists x\in \cD_n 
        : 
      b(x, c \,\omega_n^{1/2 + \xi}) \cap \ppp_n = \varnothing
      \big)
      \leq \expo{- \omega_n^{1+\xi}}
    $$
    for $n$ sufficiently large.
\end{lemma}
\begin{proof}
    We have
    \begin{align*}
    \p(
        \exists x\in  \cD_n
        :
        b(x,c\,\omega_n^{1/2+\xi})\cap \ppp_n = \varnothing
    )
    &\le 
    \p(
        \exists B\in P: B\cap \ppp_n = \varnothing
        ),
    \end{align*}
    where $P$ is any maximal packing of $\cD$ with balls $B$ of radius 
    $\frac{1}{2} c \,\omega_n^{1/2+\xi}$ centred in $\cD_n$.
    If the radius of curvature of $\cD_n$ is lower bounded by
    $ c \,\omega_n^{1/2+\xi}$ (which happen for $n$ largr enough)
    such a ball $B$ contains a ball of radius  $\frac{1}{4} c
    \,\omega_n^{1/2+\xi}$ entirely inside $\cD_n$,
   For $n$ large
    enough, any such packing contains at most $n$ balls and we have
    \[  
        \p(
            \exists x \in \cD_n
            :
            b(x, c\,\omega_n^{1/2 +\xi})\cap \ppp_n = \varnothing
        )
     \le n \expo{-\pi \tfrac{1}{4^2}} c^2\omega_n^{1+2\xi}
     \le \expo{- \omega_n^{1 + \xi}}.\qedhere
    \]
\end{proof}
\subsubsection{The size of the discs}
\label{sec:size_of_discs}
If the search cone $\cone(Z_i,q,\infty)$ does not intersect any of the previous
discs, the region which determines $R_{i+1}$ is `fresh' and $R_{i+1}$ is
independent of $\cF_i$. Lemma~\ref{lem:independent-cones} provides a condition
which guarantees independence of the search cones. To take advantage of it,
we write $\xi:=2+\sqrt 2$, and for $i\ge 0$, define the event 
\begin{equation}
    G_i:=\{\forall j\le i+1, R_j< \omega_n/ \xi\}, 
\end{equation}
Then if the event $G_i^\star:=G_i\cap \{L_{i}\ge \omega_n\}$ occurs;
for every $j\le i$, the search-cone $\cone(Z_j,q,\infty)$ does not intersect
any of the regions $\disc(Z_k,q,R_k)$, $0\le k<j$, and the corresponding
variables $(R_j,\alpha_j)$, $0\le j\le i+1$ are independent. Although it might
seem like an odd idea, $G_i^\star$ does include some condition on $R_{i+1}$;
this ensures that on $G_i^\star$, we have $L_{i+1}>L_i-2R_{i+1}>0$, so that
$i+1$ is not the last step. 
So for $x>0$ we have
\begin{align}\label{eq:bound_Ri} 
    \p(R_{i+1}> x \,|\, \cF_i,G_{i}^\star) 
    & = \p(\ppp \cap \cone(Z_i,q,x) \setminus \{Z_i\} =\varnothing 
    \,|\, \cF_i, G_{i}^\star)\I{\xi x\le \omega_n} \notag \\
    & =\expo{-Ax^2} \I{\xi x\le \omega_n},
\end{align}
where $A$ denotes the area of $\cone(z,q,1)$ which is the shaded region in
Figure~\ref{fig:next}. Indeed, conditional on $\cF_i$ and $G_i^\star$,
$|\ppp\cap \cone(Z_i,q,x) \setminus \{Z_i\}|$ is a Poisson random variable with
mean $Ax^2$ where
\begin{equation}\label{eq:def_A}
    A := 2\left( \cos \frac{\pi}{8}\sin \frac{\pi}{8} + \frac{\pi}{8} \right)
      =  \frac{\sqrt{2}}{2} + \frac{\pi}{4}.
\end{equation}
We will repeatedly use the conditioning on $G_i$ to introduce independence, and
it is important to verify that $G_i$ indeed occurs with high probability. For
$G_i$ to fail, there must be a first step $j$ for which $R_j\ge \omega_n/\xi$.
Writing $G_i^c$ for the complement of $G_i$ and defining $G_{-1}$ to be a void
conditioning: provided that $i=O(n)$ (which will always be the case in the
following)
\begin{align}\label{eq:bound_good}
    \p(G_i^c) 
    &\le \sum_{0\le j\le i+1} \prob(R_j\ge \omega_n/\xi \, | \, G_{j-1})\notag\\
    &\;\;\le \expo{\log O(n)- A \omega_n^2/\xi^2} \notag\\
    &\;\;\le \expo{-\omega_n^{3/2}}
\end{align}
for all $n$ large enough since $\omega_n\ge \log n$.

\medskip
\noindent\textbf{Remark about the notation.} 
It is convenient to work with an ``ideal'' random variable that is not
constrained by the location of the query or artificially forced to be at most
$\omega_n/\xi$, and we define $\cR$ by $\p(\cR \ge x) = \expo{-A x^2}$ for $x\ge
0$. In the course of the proof, we use multiple other such ideal random
variables, to distinguish them from the ones arising from the actual process,
we use calligraphic letters to denote them.
\subsubsection{The progress for one step}
We now focus on the distribution of the angle $\angle q Z_i Z_{i+1}$
and by extension the progress made during one step in the walk.
Let $\cone_\alpha(z,q,r)$ be the 
cone of half angle $\alpha$
with the same apex and axis as $\cone(z,q,r)$. For $S\subset \mathbb 
R^2$, let $\cA(S)$ denote its area. On the event $G_{i}^\star$, 
$Z_{i+1}\ne q$ and $\alpha_{i+1}$ is truly random and its
distribution is symmetric and given by
(see Figure~\ref{fig:x_progress}):
\begin{align}\label{eq:dist_alpha}
  \p(|\alpha_{i+1}| < x \; | \;R_{i+1}=r, \cF_i, G_i^\star )
  & = \lim_{\varepsilon\to 0}
        \frac{\cA(\cone_x(Z_i,q,r+\varepsilon)\setminus 
        \cone_x(Z_i,q,r))}{\cA(\cone(Z_i,q,r+\varepsilon)\setminus 
        \cone(Z_i,q,r))}\notag\\
  &= \lim_{\varepsilon \rightarrow 0}
            \frac{((r+\varepsilon)^2-r^2)(x + \frac{1}{2}\sin 2x )}
                 {((r+\varepsilon)^2-r^2)(
                     \frac{\pi}{8}+\frac{\sqrt{2}}{4})}\notag\\
  &= \frac{8}{\pi + 2\sqrt{2}} \left(x + \frac{\sin 2x}{2} \right).
\end{align}
So in particular, conditional on $\cF_i$ and $G_i^\star$, $\alpha_{i+1}$ is
independent of $R_{i+1}$. We will write $\alpha$ for the `ideal' angle
distribution given by \eqref{eq:dist_alpha}, and enforce that $\cR$ and
$\alpha$ be independent.

\begin{figure}[t]
\centering
\includegraphics[scale=1]{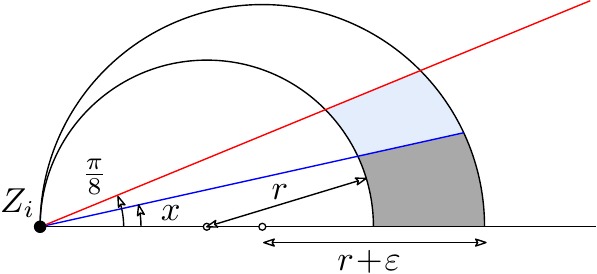}
\caption{\label{fig:x_progress}For the angle to
    be smaller than $x$ given $R_{i+1}\in[r, r+\varepsilon]$, the stopper must
    fall within the dark shaded region}
\end{figure}

\subsection{Geometric and combinatorial parameters}

In this section we will build the elements required to bound the algorithmic
complexity of the $\proc{Cone-Walk}$ algorithm. We begin by bounding the number
of \emph{steps} (or equivalently, the number of stoppers) required by the walk
process to reach the destination. We will then bound the number of vertices
\emph{visited} by the walk process, recalling that this will involve bounding
the number of \emph{intermediary vertices} within the discs $\disc(Z_i,q,R_i)$
at each step. The final part of the proof will be to bound the number of
vertices \emph{accessed} by the $\proc {Cone-Walk}$ algorithm when constructing
the sequence of stoppers and intermediary vertices. The vertices \emph{accessed}
will include all of the vertices visited, and their 1-hop neighbourhood.

\subsubsection{The maximum number of vertices accessed during a step}
\label{sec:locality}

At each step during a walk, we do not a priori access a bounded number of sites
when performing a search for the next stopper. Such a bound is important to
limit the number of neighbourhoods that may be accessed during one step, since
we note that the maximum number of vertices accessed during one step explicitly
provides an upper bound on the number of neighbourhoods accessed. A easy bound
of $\log^{1+\varepsilon} n$, for any $\varepsilon>0$ may be obtained when
considering pairs of start and destination points at least $\sqrt{\log n}$ away
from the border of $\partial\cD$. However, we opt to explicitly take care of
border effects, giving us a slightly weaker bound that can be applied
everywhere.
\begin{proposition}\label{prop:max_sites_in_one_step}  
  Let $M_\text{max}$ be the maximum number of vertices accessed during any step
  in any walk. Then 
  $$
    \p\bigg(  M_\text{max} \ge \omega_n^{3+\xi} \bigg) 
    \leq
    2\expo{-\omega_n^{1+\xi/4}}.
  $$
\end{proposition}
In the following, we note that $M_\text{max}$ is bounded by $\tau_\text{max}
\cdot \Delta_\ppp$, where $\Delta_\ppp$ gives the maximum degree of any vertex
contained within $\DT(\ppp)$ and $\tau_\text{max}$ is the maximum number sites
contained within any step in any instance of cone walk. We thus focus on
bounding $\tau_\text{max}$, and our result will follow directly from the proof
of Proposition~\ref{prop:max_degree} in Appendix~\ref{appendix:max_degree}.
\begin{lemma} \label{lem:tau_i_is_small}    
$$
  \p(\tau_\text{max} > \omega^{1+\xi}_n) \leq \expo{-\omega_n^{1+\xi/3}}.
$$
\end{lemma}
\begin{proof}
  Let $A$ be the event that the maximum disc radius for any step in any walk is
  bounded by $\tfrac{1}{2}\omega_n^{1/2+\xi}$ and let $B$ be the event that
  every ball $b(x, \tfrac{1}{2}\omega_n^{1/2+\xi})$ contains fewer than
  $\omega_n^{1+2\xi}$ points of $\ppp$, for $x\in\cD$. We have, for $n$ large enough,
  \begin{align*}
    \p(\tau_\text{max} > \omega^{1+2\xi}_n) 
    & \leq \p(\tau_{max} > \omega_n^{1+2\xi} \mid A\cap B ) + \p(A^c) + \p(B^c) \\
    & \leq \expo{-\omega_n^{1+\xi}} + \expo{-\omega_n^{1+\xi}}.
  \end{align*}
  Note that the bound on $\p(A^c)$ is implied by
  Lemma~\ref{lem:largest_empty_disc} since a large disc implicitly has a large
  empty cone. For the bound on $\p(B^c)$, we imagine splitting $\cD$ into a
  uniform grid with squares of side $\tfrac{1}{2}\omega_n^{1+\xi}$. The proof
  follows by noting that every ball of radius $\tfrac{1}{2}\omega_n^{1+\xi}$ is
  contained in a group of at most four adjacent squares, each of which must
  contain at least $\tfrac{1}{4}\omega^{1+2\xi}_n$ sites. We then use the fact
  that $n\p(\Po(\tfrac{1}{4}\omega_n^{1+\xi}) \ge \frac{1}{4}\omega_n^{1+2\xi}) 
  \leq \expo{-\omega_n^{1+\xi}}$ for $n$ large enough. We omit the details.
\end{proof}

\subsubsection{The number of steps in the walk} 
Recall that a new step is defined each time a new stopper is visited. We will
start with a first crude estimate for the decrease in distance after a given
number of steps. Note that $L_i=\|Z_iq\|$ and  $\alpha_i$ denotes the angle
between $Z_iZ_{i+1}$ and $Z_i q$. Simple geometry implies (see
Figure~\ref{fig:dist-progress}):
\begin{figure}[t]
    \centering
    \includegraphics[scale=1]{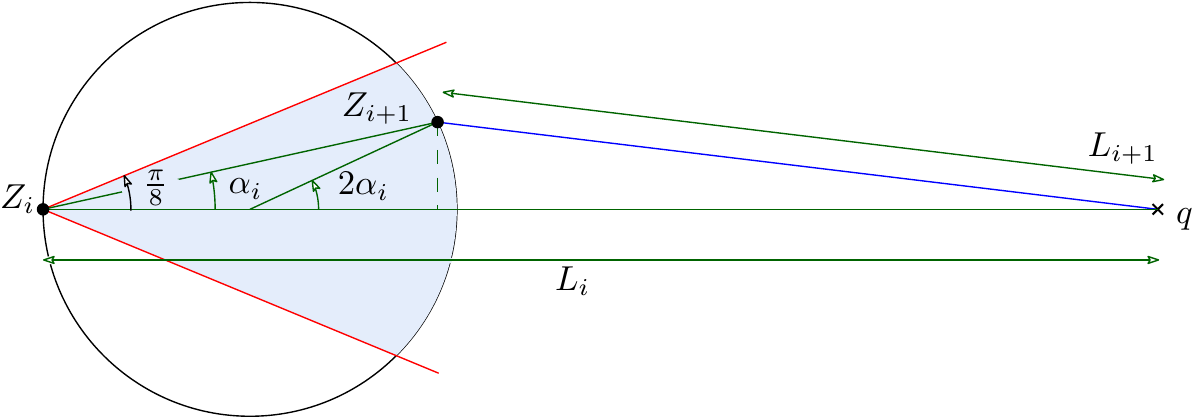}
    \caption{Computing distance progress at step $i$. 
    \label{fig:dist-progress}}
\end{figure}
\begin{align}\label{eq:bound_Li}
    L_i - R_i (1+\cos(2\alpha_i)) 
    \le L_{i+1}
    &=\sqrt{(L_i-R_i(1+\cos(2\alpha_i)))^2 + R_i^2 \sin^2(2\alpha_i) }\notag\\
    &\le L_i - R_i (1+\cos(2\alpha_i)) + 2 \frac{R_i^2}{L_i},
\end{align}
since $\sqrt{1-x}\le 1- x/2$ for any $x\in [0,1]$. {As a consequence
\begin{eqnarray}\label{eq:bound_distance1}
    L_0 - \sum_{j=0}^{i-1} R_i (1+\cos(2\alpha_i)) \le & L_i 
    & \le L_0 - \sum_{j=0}^{i-1} R_i (1+\cos(2\alpha_i)) + \frac 2{\omega_n} 
    \cdot \sum_{j=0}^{i-1} R_i^2.
\end{eqnarray}
In particular, since $\omega_n\to\infty$, after $i$ steps, the expected
distance $\e[L_i]$ to the aim $q$ should not be far from $L_0 - i \expected[\cR
(1+\cos(2\alpha))].$ Furthermore, conditional on $G_i$, and for $i$ such that
$L_i\ge \omega_n$, the summands involved in Equation \eqref{eq:bound_distance1}
are independent, bounded by $2\omega_n$ and have bounded variance, so that the
sum should be highly concentrated about its expected value
\cite{DuPa2009,BoLuMa2012a,McDiarmid1998}. In other words, one expects that for
$i$ much larger than $L_0/\e[\cR (1+\cos2\alpha)]$, it should be the case that
$L_i\le \omega_n$ with fairly high probability. Making this formal constitutes
the backbone of our proof.

\begin{lemma}          
\label{lem:distance}Let $z\in \cD$, suppose that $\ell\ge 1$ is such 
that $L_0=\|zq\|\ge (\ell +1) \omega_n$. Consider $\DT(\ppp_z)$. 
There exists a constant $\eta>0$ such that
$$\p\left(L_{0}-L_{\ell} \le \ell\, \e[\cR]/2 \right) \le \expo{-\eta 
\ell/\omega_n} +  \expo{-\omega_n^{3/2}}.$$
\end{lemma}
\begin{proof}We use the crude bounds $R_i \le L_i-L_{i+1}\le 2R_i$
(see Figure~\ref{fig:dist-progress}).   
It follows that
\begin{align*}
  \p(L_0-L_{\ell} \le \ell \,\e[\cR] /2)
  & \le \p(L_0-L_{\ell} \le \ell\, \e[\cR]/2\,|\, G_\ell) 
   + \p(G_\ell^c)\\
  & \le \p\Bigg(\sum_{j=0}^{\ell-1} R_j \le \frac \ell 2 \e [\cR] 
  \,\Bigg|\, G_\ell\Bigg)
  + \expo{-\omega_n^{3/2}},
\end{align*}
by \eqref{eq:bound_good}, since the constraint on $\ell$ imposes that
$\ell=O(\sqrt n\,)$. Now, since $L_0\ge (\ell+1)\omega_n$ and $\xi>2$, on the
event $G_\ell$, we have $L_i\ge \omega_n$ for $0\le i\le \ell$ so that
$G_\ell^\star$ occurs: conditional on $G_\ell$, the search cones do not
intersect and the random variables $R_j$, $0\le j\le \ell$ are independent and
identically distributed (see Lemma~\ref{lem:independent-cones}). Furthermore, we have
\begin{align*}
    \e[R_j\,|\, G_\ell] 
    & = \int_0^{\infty} \p\pth{R_j\ge x\,|\, G_\ell} dx\\
    & \ge \int_0^{\omega_n/\xi} \expo{-Ax^2} dx\\
    & \ge \e[\cR] - \expo{-\omega_n^{3/2}},
\end{align*}
for all $n$ large enough.
It follows that for all $n$ large enough, by Theorem~2.7 of 
\cite[][p.\ 203]{McDiarmid1998}
\begin{align*}
    \p\Bigg( \sum_{j=0}^{\ell-1} R_j \le \frac \ell 2 \e[\cR] \,\bigg|\, 
    G_\ell\Bigg) 
    & \le \p\Bigg( \sum_{j=0}^{\ell-1} (R_j - \e[R_j\,|\, G_\ell]) \le - 
    \frac \ell 3 \e[R_0\,|\, G_\ell] \,\Bigg| \, G_\ell\Bigg)\\
    & \le \expo{-\frac{t^2}{2 \ell\,
    \var(R_0\,|\, G_\ell)+2t\omega_n/3}}\qquad t=\ell\, 
    \e[R_0\,|\, G_\ell]/3\\
    & \le \expo{-\eta \ell/\omega_n},
\end{align*}
for some constant $\eta>0$ independent of $\ell$ and $n$. 
\end{proof}
The rough estimate in Lemma~\ref{lem:distance} may be significantly 
strengthened, and the very representation in \eqref{eq:bound_distance1} 
yields a bound on the number 
of search cones or steps that
are required to get within distance $\omega_n$ of the query
point $q$. (If the starting site $z$ satisfies
$L_0=\|zq\|\le \omega_n$, then this phase does not contain
any step.) 
    
\begin{proposition}
\label{pro:number_steps}Let $z\in \cD_n$, and let $\kappa(\omega_n)$ denote the
number of steps of the walk to reach a site which is within
distance $\omega_n$ of $q$ in $ \ppp_z \cup \{q\}$ when
starting from the site $z\in \ppp_z$ at distance
$L_0=\|zq\|\ge \omega_n$. Then 
$$
\p\pth{\left|\kappa(\omega_n)-\frac{L_0}{\e[\cR (1+\cos 2\alpha)]}\right|
\ge 2 \omega_n^2 \sqrt{2 L_0} + \omega_n } \le 4
\expo{-\omega_n^{3/2}}.
$$ 
\end{proposition}
\begin{proof}

We now make formal the intuition that follow Equation
\eqref{eq:bound_distance1}. We start with the upper bound. For any integer
$k\ge 0$, we have
\begin{align*}
    \p\pth{ \kappa(\omega_n)\ge k }
    & = \p\pth{L_k \ge \omega_n}  \\
    & \le \p\pth{L_k \ge \omega_n\,| \, G_k} + \p\pth{G_k^c},
\end{align*}
and since the second term is bounded in \eqref{eq:bound_good}, it now suffices
to bound the first one. However, given $G_k$ and $L_k\ge \omega_n$, the random
variables $(R_i,\alpha_i)$, $i=1,\dots, k$ are independent and identically
distributed. The only effect of this conditioning is that $R_i$ is distributed
as $\cR$ conditioned on $\cR<\omega_n/\xi$.
    
Write $X_i=R_i(1+\cos 2\alpha_i) - 2 R_i^2/\omega_n$, and note that $X_i\ge 0$
if $R_i\le \omega/\xi$. Then, from \eqref{eq:bound_distance1}, we have
\begin{align*}
    \p\pth{L_k\ge \omega_n \, | \, G_k}
    & \le \p\pth{\sum_{i=0}^{k-1} X_i \le L_0 - \omega_n \, \Bigg|\, G_{k}, 
    { L_k\ge \omega_n}}.
\end{align*}
Conditional on $ G_{k}^\star=G_k \cap \{L_k\ge \omega_n\}$, the random
variables $X_i$ are independent, $0\le X_i\le 2R_i\le \omega_n$. Furthermore,
since $X_i$ has Gaussian tails, its variance (conditional on $G_k$) is bounded
by a constant independent of $i$ and $n$. Choosing $ k_0=\lceil
(L_0+t)/\e[X_0\,|\, G_0^\star]\rceil$, for some $t<L_0$ to be chosen later, and
using the Bernstein-type inequality in Theorem~2.7 of \cite[][p.
203]{McDiarmid1998}, we obtain
\begin{align*}
    \p(L_{k_0}\ge \omega_n \,|\, G_{k_0}) 
    & \le \p\left(\sum_{i=0}^{k_0-1} (X_i - \e[X_i\,|\, G_{k_0}]) \le -t \, 
    \Bigg|\, G_{k_0}^\star\right) \\
    & \le \expo{- \frac{t^2}{2 k_0 \var(X_0\,|\, G_0^\star) + 2 \omega_n 
    t/3}}.
\end{align*}
In particular, for $t=\omega_n^3 \sqrt{L_0} $, we have for all
$n$ large enough $\p\pth{L_{k_0} \ge \omega_n\,|\, G_{k_0}} \le 
\expo{-\omega_n^{2}}$, since $L_0 \ge \omega_n$.
        
A matching lower bound on $\kappa(\omega_n)$ may be obtained similarly, 
using the lower bound on $L_{i+1}$ in Equation~(\ref{eq:bound_Li})
and following the approach we used to devise the
upper bound with $X_i'=R_i(1+\cos 2\alpha_i)$ (we omit
the details). 
It follows that, for $ k_1=\lfloor (L_0+t) / \e[X_0'\,|\, G_0^\star] \rfloor$, 
we have 
$\p(L_{k_1} \le \omega_n\,|\, G_{k_1}) \le \expo{-\omega_n^{2}}$.
    
To complete the proof,  it suffices to estimate the difference between 
$k_0$ and $k_1$. We have
\begin{align*}
    \e[X_0\,|\, G_{0}^\star]
    &= \e[R_0 (1+\cos 2\alpha_0)\, | \, G_0^\star] -
    \frac{2 \e[R_{0}^2\,|\, G_0^\star]}{\omega_n}\\
    &= \e[\cR (1+\cos 2 \alpha)] + O(1/\omega_n),
\end{align*}
and similarly, $\e[X_0'\,|\, G_0^\star]=\e[\cR(1+\cos 2\alpha)]$. It follows
that $|k_1-k_0|=O(L_0/\omega_n)$, which is not strong enough to prove the
claim. So we need to strengthen the upper bound on the second sum in the
right-hand side of
\eqref{eq:bound_distance1}. We quickly sketch how to obtain the required 
estimate. The idea is to use a dyadic argument to decompose $\kappa(\omega_n)$ 
into the number of steps to reach $L_0/2^j$, for $j\ge 1$, until one gets to 
$\omega_n$ for $j=j_0:=\lceil \log_2(L_0/\omega_n)\rceil$. For the steps $i$ 
which are taken from $Z_i$ with $L_i/L_0 \in (2^{-j}, 2^{-j+1}]$, we use 
the improved bound
$$L_{i+1}\le L_i - R_i (1+\cos 2 \alpha) + \frac 2 {L_0 2^{-j}} R_i^2.$$
Then write
$$
    \kappa(\omega_n)=\sum_{j=1}^{j_0} 
    \big[ \kappa(L_0/2^{j})-\kappa(L_0/ 2^{j-1}) \big],
$$
and observe that the $j$-th summand is stochastically dominated by 
$\kappa(L_0'/2)$ where $L_0'=L_0/2^{j-1}$. For each $j$, we define 
$k_0(j)=\lceil(L_0/2^j+t_j)/\e[X_0\,|\, G_0^\star]\rceil$ where $t_j:=\omega_n^2 
\sqrt{L_0/2^j}$ and note that
\begin{align*}
    \sum_{j=1}^{j_0} k_0(j)
    & \le \frac 1 {\e[X_0\,|\, G_0^\star]} \sum_{j=1}^{j_0} (L_0/2^j + t_j) + 
    \lceil \log_2 (L_0/\omega_n)\rceil\\
    & \le \frac {L_0} {\e[X_0\,|\, G_0^\star]} + 2 \omega_n^2 \sqrt{2 L_0}  + 
    \omega_n,
\end{align*}
for $\omega_n\ge \log n$, since $\pi L_0^2\le n$. In other words, if 
$\kappa(L_0/ 2^j)-\kappa(L_0/2^{j-1})\le k_0(j)$ for every $j$, then 
$\kappa(\omega)\le L_0/\e[X_0\,|\, G_0^\star] + 2 \omega_n^2 \sqrt{2L_0} + 
\omega_n$. 
The claim follows easily by using the union bound, where in each stretch
$[L_0/2^j, L_0/2^{j-1})$ we bound the number of steps using the previous
arguments.
\end{proof}

\begin{corollary}\label{cor:total_number_steps}
Let $z\in \cD_n$, and let $\kappa$ denote the number of steps of the walk 
to reach the objective $q$ in $\ppp_z \cup \{q\}$
when starting from 
the site $z\in \ppp_z$ at distance $L_0=\|zq\|$. Then 
$$\p\pth{\kappa > \frac{L_0}{\e[\cR (1+\cos 2\alpha)]}
     + 2 \omega_n^2 \sqrt{2 L_0} + \omega_n^3 } \le 5
  \expo{-\omega_n^{3/2}}.$$ 
\end{corollary}
\begin{proof}
It suffices to bound the number of steps $i$ such that $L_i<\omega_n$. Since
$L_i$ is decreasing, the walk only stops at most once at any given site, and the
number of steps $i$ with $L_i\le \omega_n$ is at most the number of sites lying
within distance $\omega_n$ of $q$. Recalling that $\Po(x)$ denotes a Poisson
random variable with mean $x$. We have \cite{DuPa2009}
\begin{align*}
    \p\pth{\#\{i<\kappa: L_i\le \omega_n\} \ge 2 \pi \omega_n^2}
    & \le \p\pth{\text{Po}(\pi \omega_n^2) \ge 2 \pi \omega_n^2}\notag\\
    & \le \expo{-\pi \omega_n^2/3}.
\end{align*}
The claim then follows easily from the upper bound in
Proposition~\ref{pro:number_steps}.
\end{proof}
\subsubsection{The number of vertices in the discs}
\label{sec:visited_points}
\noindent
We now bound the total number of vertices visited, which we recall is
exactly the the number of points in $\ppp_n$ falling within union of all of the
discs in the walk.
Proposition~\ref{pro:number_steps} will be the key to analysing the path
constructed by the walk: representations based on sums of random variables
similar to the one in \eqref{eq:bound_Li} may be obtained to upper bound the
number of steps and intermediate steps visited by the walk (which is an upper
bound on the vertices visited by the path), and also the sum of the length of
the edges.
\begin{proposition}
\label{pro:bound_substeps}Let $K=K(z)$ be the
number of vertices visited by the walk starting from a given
site $z$ with $L_0=\|zq\|$. Then, for all $n$ large enough, 
$$\p\left(K\ge \frac{L_0}{\e[\cR(1+\cos 2\alpha)]} \cdot \frac
{\pi-A}A + \sqrt{L_0} \omega_n^4 + \omega_n^3 \right)\le
7 \expo{-\omega_n^{3/2}}.$$
\end{proposition}

\begin{proof}
There are two contributions to $K-\kappa(\omega_n)$: first the number of
intermediate steps which lie at distance greater than $\omega_n$ from $q$, and
all the sites which are visited and lie within distance $\omega_n$ from $q$.
Let $K=K_1+K_2$ where $K_1$ and $K_2$ denote these two contributions,
respectively.
By the proof of Corollary~\ref{cor:total_number_steps}, we have 
\begin{align}\label{eq:bound_firststeps}
\p\pth{K_2 \ge 2 \pi \omega_n^2}
& \le \expo{-\pi \omega_n^2/3}.
\end{align}
To bound $K_1$, observe that the monotonicity of $L_i$ implies that $K_1$
counts precisely the number of intermediate steps before reaching the disc of
radius $\omega_n$ about $q$. Observe that if $L_0<\omega_n$, $K_1=0$, so we may
assume that $L_0\ge \omega_n$. Recall that $\tau_i$ denotes the number of
\emph{intermediate} points at the $i$-th step. Note that the intermediate
points counted by $\tau_i$ all lie in $\disc(Z_i,q,R_i)\setminus
\cone(Z_i,q,R_i)$, and given the radius $R_i$, $\tau_i$ is stochastically
bounded by a Poisson random variable with mean $(\pi-A)R_i^2$. Furthermore, on
the event $G_{\kappa(\omega_n)}$, the random variables $R_i, i=0,\dots,
\kappa(\omega_n)$ are independent. Also, by Lemma~\ref{lem:independent-discs}
the regions $\disc(Z_i,q,R_i)\setminus \cone(Z_i,q,R_i)$, $i\ge 0$, are
disjoint so that the random variables $\tau_i$, $i=0,\dots, \kappa$ are
independent given $R_i$, $i=0,\dots, \kappa$.

Let $\widetilde \cR_i$, $i\ge 0,$ be a sequence of i.i.d.\ random variables
distributed like $\cR$ conditioned on $\cR\le \omega_n/\xi$ and given this
sequence, let $\tilde \tau_i$, $i\ge 0$, be independent distributed like
$\Po((\pi-A)\widetilde \cR_i)$. As a consequence of the previous arguments, for
$k=k_0+ 2 \omega_n^2 \sqrt{2 L_0} + \omega_n$ 
with $k_0=\lceil L_0 / \e[\cR(1+\cos 2 \alpha)]\rceil$, we have
\begin{align}\label{eq:bound_K1}
    \p\pth{K_1 \ge \ell }
    &\le \p\Bigg(\sum_{i=0}^{(k-1)\wedge 
    \kappa(\omega_n)} \tau_i \ge \ell \Bigg) + 
    \p\pth{\kappa(\omega_n) \ge k_0+2 \omega_n^2 \sqrt{2 L_0} + 
    \omega_n}\notag\\
    &\le \p\left(\sum_{i=0}^{k-1} \tilde \tau_i \ge \ell \right) + \p(G_k^c) +
    \p\pth{\kappa(\omega_n) \ge k_0+2 \omega_n^2 \sqrt{2 L_0} + 
    \omega_n} \notag \\
    &\le \p\left(\sum_{i=0}^{k-1} \tilde \tau_i \ge \ell \right) + 
    5\expo{-\omega_n^{3/2}},
\end{align}
by \eqref{eq:bound_good} and Proposition~\ref{pro:number_steps}. 
		
We now bound the first term in \eqref{eq:bound_K1}. Note that $i\ge 0$, we have
$$\e[\tilde \tau_i] = (\pi-A) \e[\widetilde \cR_i^2] \le
(\pi-A)\e[\cR^2]=\frac{\pi-A}A=:\gamma,$$ and we expect that $\sum_{i=0}^{k-1}
\tilde \tau_i$ should not exceed its expected value, $k\gamma$ by much. Write
$\ell=k\gamma + t$, for some $t$ to be chosen later. For the sum to be
exceptionally large either the radii of the search discs are large, or the
discs are not too large but the number of points are:
\begin{align}\label{eq:bound_sum_tau}
    \p\left(\sum_{i=0}^{k-1} \tilde \tau_i\ge k \gamma + t \right)
    &= \p\left(\Po\bigg((\pi-A)\sum_{i=0}^{k-1} \widetilde \cR_i^2 \bigg) 
    \ge k \gamma + t \right)\notag\\
    &\le \p\left(\Po\Big(k \gamma + \frac t 2\Big) \ge k \gamma +t \right) + 
    \p\left(\sum_{i=0}^{k-1} \widetilde \cR_i^2 
    \ge \frac{k \gamma + t/2}{\pi-A} \right). 
\end{align}
The first term simply involves tail bounds for Poisson random variables. For
$t=\sqrt{L_0}\omega_n^4$, we have
\begin{align*}
    \p\left(\Po\Big(k\gamma + \frac t 2\Big) \ge k \gamma +t \right) 
    &\le \expo{-\frac{(t/2)^2}{3( k \gamma + t/2)}}\\
    &\le \expo{-\omega_n^{3}},
\end{align*}
for $n$ large (recall that we can assume here that $L_0\ge \omega_n$.) The
second term in \eqref{eq:bound_sum_tau} is bounded using the same technique as
in the proof of Proposition~\ref{pro:number_steps} above. Since we have $0\le
\widetilde \cR_i^2\le \omega_n^2$ and $\e[\widetilde \cR_i^2]\le 1/A$, we
obtain for some positive constant $c$, 
\begin{align*} 
    \p\left(\sum_{i=0}^{k-1}
    \widetilde \cR_i^2 \ge \frac{k \gamma + t/2}{\pi-A} \right) 
    &\le \p(G_k) +
    \expo{-8 \frac{k}{A^2 \omega_n^4}}
    \le \expo{- \frac{c t^2}{k
    \var(\cR)+ \omega_n^2 t}}. 
\end{align*} Recalling that $L_0\ge \omega_n$,
yields 
$$
    \p\left(\sum_{i=0}^{k-1} \tilde \tau_i\ge k\gamma + \omega_n^4
    \sqrt{L_0}\right)\le \expo{-\omega_n^2},
$$ 
for all $n$ large enough, which
together with \eqref{eq:bound_K1} proves that $\p\pth{K_1\ge k \gamma +
\omega_n^4 \sqrt{L_0}} \le 6 \expo{-\omega_n^{3/2}}$. Using
\eqref{eq:bound_firststeps} readily yields the claim.
\end{proof}

\subsubsection{The length of $\proc{Simple-Path}$}
When using $\proc{Competitive-Path}$, the path length is deterministically
bounded by the length of the walk. However, for $\proc{Simple-Path}$, the path
length is dependent on the configuration of the points inside the discs. We will
show that with strong probability and as long as the walk is sufficiently long,
the path length given by $\proc {Competitive-Path}$ is no better in an
asymptotic sense than that given by $\proc{Simple-Path}$.
\begin{proposition}
\label{pro:walk_length}For $z\in \cD$, let $\Lambda=\Lambda(z)$ be the sum of
the lengths of the edges of $\DT(\ppp_z)$ used by $\proc{Simple-Path}$ given a
walk with objective $q$
and starting from $z$ such that $L_0=\|zq\|$. Then,
\begin{equation}\label{eq:bound_length}
    \p\Big(\Lambda \ge c L_0 + (3\sqrt{L_0}+1) \omega_n^4 \Big) \le 
    8\expo{-\omega_n^{3/2}}\qquad \text{where}\qquad c:=\frac{22 \pi - 4 \sqrt 
    2}{2+3\pi+8\sqrt 2}.
\end{equation}
\end{proposition}
\begin{proof}
Write $\lambda_i$ for the sum of the lengths of the edges used by the walk to
go from $Z_i$ to $Z_{i+1}$. So $\Lambda=\sum_{i=0}^{\kappa-1} \lambda_i$. Our
bound here is very crude: all the intermediate points remain in
$\disc(Z_i,q,R_i)$, and given $R_i$, we have $\lambda_i\le (1+\tau_i) \cdot 2
R_i$. Again, on $G_k$ the cones do not intersect provided that $L_k\ge
\omega_n$, and by Lemma~\ref{lem:independent-discs} the random variables
$\lambda_i$, $0\le i<k$ are independent. We use once again the method of
bounded variances (Theorem~2.7 of \cite{McDiarmid1998}).

We decompose the sum into the contribution of the steps before 
$\kappa(\omega_n)$ and the ones after:
\begin{align}\label{eq:Lambda_decomp}
    \p\pth{\Lambda \ge x+t}
    & \le \p\Bigg(\sum_{i=0}^{(k-1) \wedge \kappa(\omega_n)} 
    \lambda_i \ge x \Bigg)+ 
    \p\pth{\kappa(\omega_n)\ge k} + 
    \p\Bigg(\sum_{i=\kappa(\omega_n)}^{\kappa-1} \lambda_i \ge t\Bigg).
\end{align}
For $i\ge \kappa(\omega_n)$, $\disc(Z_i,q,R_i)$ is contained in $
b(q,\omega_n)$, { the disc of radius $\omega_n$ around $q$}, and the
contribution of the steps $i\ge \kappa(\omega_n)$ is at most $2\omega_n
|{\mathbf \Phi} \cap b(q,\omega_n)|$. In particular
\begin{align*}
    \p\Bigg(\sum_{i=\kappa(\omega_n)}^{\kappa-1} \lambda_i \ge t \Bigg)
    & \le \p\left(2\omega_n\Po(\pi\omega_n^2) \ge t\right)\\
    & \le \expo{-\omega_n^{3/2}},
\end{align*}
for all $n$ large enough provided that $t\ge 4 \pi \omega_n^3$. To make sure
that the second contribution in \eqref{eq:Lambda_decomp} is also small, we rely
on Proposition~\ref{pro:number_steps} and choose $k=\lceil L_0/\e[\cR(1+\cos
2\alpha)] + 2\omega_n^2 \sqrt{2L_0}+\omega_n\rceil$ so that
$\p(\kappa(\omega_n)\ge k)\le 4 \expo{-\omega_n^{3/2}}$.
		
Finally, to deal with the first term in \eqref{eq:Lambda_decomp}, we note that
on $G_k$, the random variables $\lambda_i$, $i=0,\dots, k+1$ are independent
given $R_i$, $i=0,\dots, k+1$. Let $\widetilde \cR_i$, $i=0,\dots, k-1$ be
i.i.d.\ copies of $\cR$ conditioned on $ \cR\le \omega_n/\xi$; then let $\tilde
\tau_i$ be independent given $R_i$, $i=0,\dots, k+1$, and such that $\tilde
\tau_i=\Po((\pi-A) \widetilde \cR_i)$; finally, let $\tilde \lambda_i = 2 \cR_i
(1+\tilde \tau_i)$. We choose $x=k \e[\tilde \lambda_0] + y$ with
$y=\sqrt{L_0}\omega_n^4$. Using arguments similar to the ones we have used in
the proofs of Propositions~\ref{pro:number_steps} and~\ref{pro:bound_substeps},
we obtain

\begin{align}\label{eq:length_outside}
    \p\Bigg(\sum_{i=0}^{(k-1) \wedge \kappa(\omega_n)} \lambda_i \ge x\Bigg)
    & \le \p\left( \sum_{i=0}^{k-1} \tilde \lambda_i \ge x \right) + \p(G_k^c)\\
    & \le \expo{- \frac{y^2}{2k \var(\tilde \lambda_0)+ 2 \omega_n^2 
    y/3}} + \p\pth{\exists i<k: \tilde \lambda_i \ge \omega_n^2} + 
    \p\pth{G_{k}^c}.\notag
\end{align}
To bound the second term in the right-hand side above, observe that for $i<k$,
we have, for any $x>0$,
\begin{align*}
    \p\big(\tilde \lambda_i \ge x^2 \big)
    & \le \p\big((1+\tilde \tau_i)2 \widetilde \cR_i \ge x^2\big)\\
    & \le \p\big((1+\tilde \tau_i)2 \cR_i \ge x^2 \,\big|\, 
        2\widetilde\cR_i\le x\big) + \p\big(2 \widetilde \cR_i \ge x\big) \\
    & \le \p\big(1+\tilde \tau_i \ge x \,\big|\, 2\widetilde \cR_i \le x\big) + 
        \p\big(2 \widetilde \cR_i \ge x\big) \\
    & \le \p\pth{\Po((\pi-A)x^2/4) \ge x-1} + \expo{-A x^2/4}\\
    & \le 2\expo{-\eta x^2},
\end{align*}
for some constant $\eta>0$ and all $x$ large enough. It follows immediately
that $\var(\tilde \lambda_0)<\infty$ and that, $n$ large enough,
\begin{align}\label{eq:bound_lambdai}
    \p\big(\exists i<k: \tilde \lambda_i \ge \omega_n^2\big)
    & \le k \expo{-\eta \omega_n^2}\notag\\
    & \le \expo{-\omega_n^{3/2}}.
\end{align}
Going back to \eqref{eq:length_outside}, we obtain
\begin{align*}
    \p\Bigg(\sum_{i=0}^{(k-1) \wedge \kappa(\omega_n)} \lambda_i \ge x\Bigg)  
    & \le 3 \expo{-\omega_n^{3/2}}, 
\end{align*}
since here, we can assume that $L_0\ge \omega_n$ (if this is not the case, the
points outside of the disc of radius $\omega_n$ centred at $q$ do not
contribute). Putting the bounds together yields 
$$\p\pth{\Lambda \ge x+t} \le 8 \expo{-\omega_n^{3/2}},$$
and the claim follows by observing that for 
$$
    c:=\frac{\e[2(1+\Po((\pi-A)\cR^2))\cR]}{\e[\cR(1+\cos 2 \alpha)]}
    = \frac{2 \e[\cR] + \e[2 (\pi-A) \cR^3]}{\e[\cR(1+\cos 2 \alpha)]},
$$
it is the case that $cL_0+(3\sqrt{L_0}+1) \omega_n^4\ge x+t$ for all $n$ large
enough. Simple integration using the distributions of $\cR$ and $\alpha$ then
yields the expression in \eqref{eq:bound_length}.
\end{proof}

\subsection{The number of sites accessed}
\label{sec:accessed_points}
In this section, we bound the total number of sites accessed by the cone-walk
algorithm, counted with multiplicity. We note a point is accessed at step $i$ if
it is the endpoint of an edge whose other end lies inside the disc $D_i$. A
given point may be accessed more than once, but via different edges, so we will
bound the number of edges such that for some $i$, one end point lies inside
$D_i$ and the other outside; we call these \emph{crossing edges}. (Note that the
number of crossing edges does not quite bound the complexity of the algorithm
for the complexity of a given step is not linear in the number of accessed
points; however it gives very good information on the amount of data the
algorithm needs to process.)

\begin{proposition}\label{pro:bound_algo}
  Let $A = A(z)$ be the number of sites in $\ppp$ accessed by the cone walk
  algorithm (with multiplicity) when walking towards $q$ from $z$. Then there
  exists a constant $c>0$ such that
  $$
    \p\bigg(A(z) > c \,L_0 + 4\left(\sqrt{L_0}+1\right)\omega_n^6\  \bigg) 
    \leq
    3\expo{- \omega_n^{5/4}}.
  $$
  For $n$ sufficiently large.
\end{proposition}

\begin{proof}
In order to bound the number of such edges, we adapt the concept of the
\emph{border point} introduced by \citet{BoDe2006} to bound the stabbing number
of a random Delaunay triangulation. For $B\subseteq \cD_n$ and a point
$x\in \cD_n$, let $\|xB\|:=\inf\{\|xy\|: y\in B\}$ denote the distance from $x$
to $B$.

We consider the walk from $z$ to $q$ in $\cD_n$, letting
$$
  W=\bigcup_{i=1}^\kappa D_i\qquad \text{and}\qquad W^{\odot}
  :=
  b(\mathbf{w},2 \max\{\|zq\|, \omega^5_n\}),
$$ 
where $\mathbf{w}$ denotes the centroid of the segment $zq$ and we recall that
$b(\mathbf{x}, r)$ denotes the closed ball centred at $\mathbf{x}$ of radius
$r$. Then, for $x\in W^\odot$, let $C$ be the disc centred at $x$ and with
radius $\min\{\|xW\|,\|x\partial W^\odot\|\}$. Partition the disc $C$ into 8 
isometric cone-shaped sectors (such that one of the separation lines is
vertical, say) truncated to a radius of $\sqrt{3}/2$ times that of the outer
disc (see Figure~\ref{fig:border-point}). We say that $x$ is a border point if
one of the  8 
is a border point then there is no Delaunay edge between $x$ and a point lying
outside $C$, since a circle trough $x$ and $y\not\in C\subset W^\odot\setminus
W$ must entirely enclose at least one sector of $C$ (see dotted circle in
Figure~\ref{fig:border-point}). Thus if $x$ has a Delaunay edge with extremity
in $W$, then $x$ must be a border point.
\begin{figure}[t]
  \centering
  \includegraphics[scale=.7]{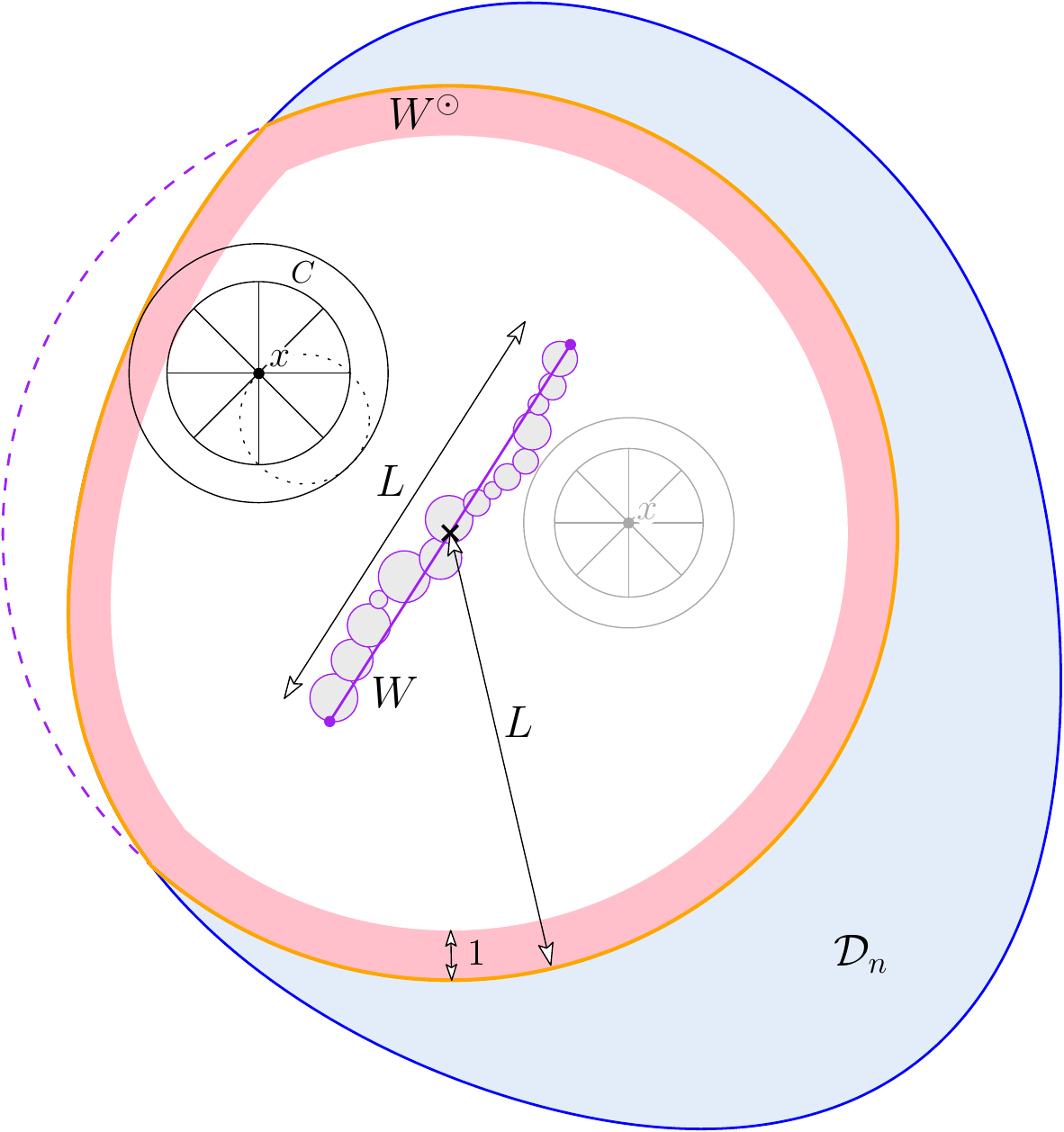}
  \caption{
    \label{fig:border-point} For the proof of Proposition~\ref{pro:bound_algo}.
  }
\end{figure}
The connection between border points and the number of crossing edges can be
made via Euler's relation, since it follows that a crossing edge is an edge of
the (planar) subgraph of $\DT(\ppp)$ induced by the points which either lie
inside $W$, are border points, or lie outside of $W^\odot$ and have a neighbour
in $W$. Let $B_W$ denote  of set of border points, $E_W$ the collection of
crossing edges, and $Y_W$ the collection of points lying outside of $W^\odot$
and having a Delaunay neighbour within $W$. Then
\begin{equation}\label{eq:euler}
A(z)\le |E_W| \le 3 (|W\cap \ppp| + |B_W|+ |Y_W|).
\end{equation}
Proposition~\ref{pro:bound_substeps}
bounds $|W\cap \ppp|$, as this is exactly the set of
\emph{visited vertices}. 
Lemmas~\ref{lem:bound_Yw} and~\ref{lem:bound_Bw} 
bounding $|B_W|$ and $|Y_W|$ complete the proof.
\end{proof}

\begin{lemma}\label{lem:bound_Yw}
For all $n$ large enough, we have
$$
  \p\bigg(|Y_W|\ge 10 \max\{ L_0,  \omega_n^5 \}\bigg) 
  \le 2\expo{-\omega_n^{5/4}}.
$$
\end{lemma}
\begin{proof}
  Heuristically, our proof will follow from the fact that,
  with high probability, a Delaunay edge away
  from the boundary of the domain is not long enough to span the distance
  between a point within the walk, and a point outside of $W^\odot$. 
  Unfortunately our proof is complicated by points on the walk
  which are very close to the boundary of the domain, since in this case, those
  points might have `bad' edges which are long enough to escape $W^\odot$. To
  deal with this, we will take all points in the walk that are close to the
  border, and imagine that every Delaunay edge touching one of these points is
  such a `bad' edge. The total number of these edges will be bounded by the
  maximum degree.

  To begin, we give the first case. Consider an arbitrary point $x\in W\cap\ppp$
  that is at least $\omega_n^{-3}$ away from the boundary of $\cD_n$. Suppose
  this point has a neighbour outside of $W^\odot$, then its circumcircle
  implicitly overlaps an unconditioned region of $\cD_n$ with area at least
  $c\,\omega_n^{-3} \,
  \omega_n^{5} = c\, \omega_n^2$ (for $c>0$ a constant depending on the shape
  of the domain). The probability that this happens for $x$ is thus at most
  $\expo{-\omega_n^2}$. Now note that there are at most $2n$ points in $\ppp$ with
  probability bounded by $\expo{-\omega_n^2}$ and at most $4n^2$ edges between
  points of $x\in W\cap \ppp$ and $x\in \{ W^\odot\}^c \cap \ppp$. By the union
  bound, the probability that any such edge exists is at most
  $$(4n^2)\expo{-c\,\omega_n^2} + \expo{-\omega_n^2} \leq \expo{-\omega_n^{3/2}}.$$

  For the second case, we count the number of points within $\omega_n^{-3}$ of
  the boundary of the domain. Using standard arguments, we have that there are
  no more than $10\max\{L, \omega_n^5 \}\cdot \omega_n^{-3}$ such points, with
  probability at least $\expo{-\omega_n^2}$. Each of these has at most
  $\Delta_\ppp$ edges that could exit $W^\odot$, where $\Delta_\ppp$ is the
  maximum degree of any vertex in $\DT(\ppp)$, which is bounded in
  Proposition~\ref{prop:max_degree}. Thus, the number of such bad edges is at
  most $10\max\{L, \omega_n^5 \}\omega_n^{-3} \cdot
  \omega_n^{3}$ with probability at least $\exp(-\omega_n^{2}) +
  \exp(-\omega_n^{5/4})$
\end{proof}
} 

\begin{lemma}\label{lem:bound_Bw}
  For all $n$ large enough, and universal constant $C>0$,
  $$
    \p\bigg(
      |B_W|\ge C \max\{ L_0,  \omega_n^6 \}
    \bigg)
    \le 
    2\expo{-\omega_n^{3/2}}.
  $$
\end{lemma}
\begin{proof}
  Since $|B_W|$ is a sum of indicator random variables, it can be bounded using
  (a version of) Chernoff--Hoeffding's method. The only slight annoyance is that
  the indicators $\I{x\in B_W}, x\in \ppp \cap W^\odot$ are not independent.
  Note however that $\I{x\in B_W}$ and $\I{y\in B_W}$ are only dependent if the
  discs used to define membership to $B_W$ for $x$ and $y$ intersect. There is a
  priori no bound on the radius of these discs, and so we shall first discard
  the points $x\in \ppp$ lying far away from $\partial \Wo$ and $W$. More
  precisely, let $B_W^\star$ denote the set of border points lying within
  distance $\omega_n$ of either $W$ or $\partial \Wo$, and
  $B_W^\bullet=B_W\setminus B_W^\star$. Observe now that we may bound
  $|B_W^\bullet|$ directly using Lemma~\ref{lem:largest_empty_disc}, since a
  point is only a border point if one of its cones is empty, and each 
  such empty cone contains a large empty circle. So for sufficiently large $n$,
  \begin{align}\label{eq:bound_Bbul}
    \p\big( |B_W^\bullet| \ne 0 \big) \le \expo{-\omega_n^{3/2}}.
  \end{align}
Bounding $|B_W^\star|$ is now easy since the amount of dependence in the family 
$\I{x\in B_W}, x\in \ppp\setminus W$ is controlled and we can use the inequality 
by Janson \cite{Janson2004b,DuPa2009}. We start by bounding the expected value 
$\e\,|B_W^\star|$. Note that for a single point $x\in \ppp_n$, by definition the disc used to 
define whether $x$ is a border point does not intersect $W$ and stays entirely within $\cD$, so
\begin{align}\label{eq:prob_border_point}
  \p_x(x\in  B_W^\star)=\p(x\in B_W^\star)
  & \le 24 \expo{-\frac \pi {32} \min\{\|xW\|, \|x\partial \Wo\|\}^2 } \notag \\
  & \le 24 \expo{-\frac \pi {32} \|xW\|^2} 
  + 24 \expo{-\frac \pi {32} \|x \partial \Wo\|^2}
\end{align}
and $\ppp_n$ is unconditioned in $\cD\setminus W$. 
Partition $\cD\setminus W$ into disjoint sets as follows:
$$\cD\setminus W = \bigcup_{i=0}^\infty U_i$$
where $U_i := \{x\in\cD \;:\;i\leq\|xW\|< i+1 \}.$ Similarly, the sets 
$U_i':= \{x \in \Wo : i\le \|x \partial \Wo\| <i+1\}$ form a similar partition 
for $\Wo$. 
Writing $\lambda$ for the 2-dimensional Lebesgue measure and using \eqref{eq:prob_border_point}
above, we have
\begin{align*}
    \e|B_W^\star| 
    & = \e\left[\sum_{x\in \ppp }\I{x\in \Wo \setminus W}\; \I{x\in B^\star_W} \right]\\
    &= \int_{\Wo \setminus W} \p_x(x\in B^\star_W)\, \lambda(dx)\\
    &\leq \sum_{i=0}^\infty 
    \int_{U_i } 24\expo{-\pi i^2/32} \,\lambda(dx) + 
    \sum_{i=0}^\infty \int_{U_i'} 24 \expo{-\pi i^2/32}\,\lambda(dx)\\
    &= 24\sum_{i=0}^\infty (\lambda( U_i) + \lambda(U_i') )\; \expo{-\pi i^2/32}.
\end{align*}
We may now bound $\lambda(U_i)$ and $\lambda(U_i')$ as follows. Recall that $W$ is a union of discs
$W = \cup_i D_i$. We clearly have that
\begin{align*}
    U_i &\subseteq \bigcup_{j=0}^{\kappa-1}  \big\{ x\in\cD \;:\; i\leq\|xD_j\|<i+1 \big\}.
\end{align*}
Note that
\begin{align}\label{eq:walk_boundary}
    \lambda\big(\big\{ x\in\cD \;:\; i\leq\|xD_j\|<i+1 \big\}\big)
    &\leq \pi ((R_j+i+1)^2 - (R_j+i)^2)\\
    &= \pi (2(R_j + i) + 1).
\end{align}
So, assuming there are $\kappa$ steps in the walk we get 
\begin{align*} 
    \lambda(U_i) 
    &\leq \sum_{j=0}^{\kappa-1} \pi (2(R_j + i) + 1) 
    = 2\pi \sum_{j=0}^{\kappa-1} R_j + \pi (i+1)\kappa.
\end{align*}
Regarding $\lambda(U_i')$, note first that $\Wo$ is convex for it is the intersection of two convex
regions. It follows that its perimeter is bounded by $4\pi \max\{\|zq\|, \omega_n\}$, so that 
$\lambda(U_i')\le 4 \pi \max\{\|zq\|, \omega_n\}$ for every $i\ge 0$. 
It now follows easily that there exist universal constants $C,C'$ such that  
\begin{align*}
\e\,|B^\star_W|=  \e\, \e \big[ |B_W^\star| \; \big| \; R_i, i\ge 0\big] 
  &\le C \,\e \left[ \sum_{j=0}^{\kappa-1} R_j + \kappa + \max\{\|zq\|, \omega_n\}\right]\\
  &\le C' \max\{\|zq\|, \omega_n\}.
\end{align*}

For the concentration, we use the fact that if $\|xW\|,\|x\partial \Wo\| \le \omega_n$ then 
the chromatic number $\chi$ of the dependence graph of the family $\I{x\in B^\star_W}$ is 
bounded by the maximum number of points of $\ppp_n$ contained in a disc of radius $2\omega_n$. 
We then have
\begin{align*}
\p(\chi \ge 8\pi\omega_n^2)
&\le \p(\exists x\in \cD: \ppp_n\cap b(x,2\omega_n))\\
&\le \e\left[ \sum_{x\in \ppp_n} \p_x(|b(x,2\omega_n)\cap \ppp_n|\ge 8\pi \omega_n^2)\right]\\
&\le \e\left[ \sum_{x\in \ppp_n} \p(\Po(4\pi\omega_n^2)\ge 8\pi \omega_n^2)\right]\\
&\le \expo{-\omega_n^2},
\end{align*}
for all $n$ large enough, using the bounds for Poisson random variables we have 
already used in the proof of Corollary~\ref{cor:total_number_steps}.  
Let
\begin{equation*}
  W^\partial 
  :=
  \big\{ 
    x\in W^\odot 
    \; \big| \;
    \max  \{ \|xW\|, \|x\partial W^\odot \|\} 
    \leq \omega_n 
  \big\}.
\end{equation*}
Following Equation~\eqref{eq:walk_boundary} and by the convexity of $W^\odot$,
there exists a universal constant $C''$ such that
\begin{equation*}
  \p \bigg( | \ppp_n \cap W^\partial | \ge C'' \max\{L, \omega_n^5\} \cdot \omega_n^2 \Bigg)
  \le \expo{-\omega_n^2}.
\end{equation*}
By Theorem~3.2 of \cite{DuPa2009}, we thus obtain for $t>0$,
\begin{align*}
& \p(|B^\star_W|\ge \e \,|B^\star_W| + t) \\
& \le \e\left[\expo{-\frac{2t^2}{\chi \cdot |\ppp_n \cap W^\partial|}} \right]\\
& \le \expo{- \frac{t^2}{8\pi\, C'' \max\{L_0, \omega^5_n\} \cdot \omega_n^4}} 
  + \p\Big(\chi \ge 8\pi \omega_n^2\Big) 
+ \p\Big(|\ppp_n \cap W^\partial| > C'' \max\{L_0, \omega^5_n\} \cdot \omega_n^2\Big) \\
& \le \expo{- \frac{t^2}{8\pi\, C'' \max\{L_0, \omega^5_n\} \cdot \omega_n^4}} + 2\expo{-\omega_n^2}.
\end{align*}
The result follows for $n$ sufficiently large by choosing 
$ t :=  C''\left(L_0 +\omega^6_n\right)$.
\end{proof}


\subsection{Algorithmic complexity}\label{sec:algorithmic_complexity}
Whilst we have given explicit bounds on the number of sites in $\ppp_n$\
that may be accessed by an instance of $\proc{Cone-Walk}$, we recall that the
complexity of the algorithm $\proc{Cone-Walk}(z,q)$ does not follow directly.
This is because the $\proc {Cone-Walk}$ algorithm must do a small amount of
computation at each step in order to compute the vertex which should be chosen
next. We proceed by defining a random variable $T(z,q)$, which will denote the
number of operations required by $\proc{Cone-Walk}(z,q)$ in the RAM model of
computation given an implementation based upon a priority queue. We conjecture
that the bound for Proposition~\ref{prop:complexity} given in this section is
not tight, and that the algorithmic complexity is fact be bounded by $O(L_0 +
\omega_n^4)$. Unfortunately the dependency structure in algorithms of this type
makes such bounds difficult to attain.
\begin{proposition}  \label{prop:complexity}
  Let $T(z,q)$ be the number of steps required by the $\proc{Cone-Walk}$
  algorithm to compute the sequence of stoppers given by the cone-walk process
  between $z$ and $q$ along with the path generated by $\proc{Simple-Path}$
  in the RAM model of computation.
  Let $c_1>0$ be an implementation-dependent constant, then for $n$
  large enough, we have
  \begin{align*}
    \p\bigg(
      T(z,q) > c_1 \cdot L_0 \log  \omega_n + \omega_n^4
    \bigg) \leq
    10\expo{-\omega_n^{3/2}}.
  \end{align*}
\end{proposition}
\begin{proof}
Define $M_i$ to be the number of sites accessed during the $i$'\textit{th}
step of the algorithm. We fix $z,q$ and write $T$ to denote $T(z,q)$ for
brevity. From Section~\ref{sec:complexity}, we know that for sufficiently large
$n$, we may choose a constant $c_2$ such that
\begin{align}\label{eqn:complexity}
  T \leq c_2 \sum_{i=0}^{\kappa-1}\tau_i \log(M_i).
\end{align}
So that to bound $T$, it suffices to bound the sum in \eqref{eqn:complexity}.
We have

\begin{align*}
  \begin{split}
    \p\bigg(
      \sum_{i=0}^{\kappa-1}\tau_i & \log(M_i) 
      \ge
      15(L_0 + {\omega_n}^3)\log(\omega_n)
    \bigg)\\
     & \le 
    \p\bigg(
      \sum_{i=0}^{\kappa-1}\tau_i\log(M_i) 
      \ge
      3(L_0 + {\omega_n}^3)\log(\omega_n^5) 
      \mid M_\text{max} < \omega_n^5 
    \bigg)  + 
    \p\bigg( 
     M_\text{max} \ge \omega_n^5 
    \bigg) \\ 
    &\leq
    \p\bigg(
      \sum_{i=0}^{\kappa-1}\tau_i
      \ge
      3(L_0 + {\omega_n}^3)
    \bigg) + 
    \expo{-\omega_n^{3/2}}.
  \end{split}
\end{align*}
The first term of which is bounded by Proposition~\ref{pro:bound_substeps},
and the bound on $M_\text{max}$ comes directly from the proof of
Proposition~\ref{prop:max_sites_in_one_step}.
\end{proof}

\section{Relaxing the model and bounding the cost of the worst query}
\label{sec:relax_query}
\subsection{About the shape of the domain and the location of the aim}
The analysis in Section~\ref{sec:analysis} was provided given the assumption
that $q$ was the centre of a disc containing $\ppp$ for clarity of exposition.
We now relax the assumptions on both the shape of the domain $\cD$ and the
location of the query $q$. Taking $\cD$ to be a disc with $q$ at its centre
ensured that $\disc(z,q,r)$ was included in $\cD$ for $r \leq \|zq\|$ and thus
the search cone and disc were always entirely contained in the domain $\cD$. If
we now allow $q$ to be close to the boundary, it may be that part of the search
cone goes outside~$\cD$.

To begin with, we leave $\cD$ unchanged and allow $q$ to be any point in $\cD$.
Given any point $z\in\cD$, the convexity of $\cD$ ensures that the line segment
$zq$ lies within $\cD$. Furthermore, one of the two halves of the disc of
diameter $zq$ is included within $\cD$. Thus for any $z,q\in \cD$ and $r\in
\reals$ the portion of $\cone(z,q,r)$ (resp. $\disc(z,q,r)$) within $\cD$ has
an area lower bounded by half of its actual area (including the portion outside
$\cD$). Since the distributions of all of the random variables rely on
estimations for the portions of area of $\cone(z,q,r)$ or $\disc(z,q,r)$ lying
inside $\cD$, we have the same order of magnitude for $\kappa$, $K$, $\Lambda$
and $T$, with only a degradation of the relevant constants. The proofs
generalise easily, and we omit the details. (Note however, that upper and lower
bounds in an equivalent of Proposition~\ref{pro:number_steps} would not match
any longer.)

The essential property we used above is that a disc with a diameter within
$\cD$ has one of its halves within $\cD$. This is still satisfied for smooth
convex domains $\cD$ and for discs whose radius is smaller than the minimal
radius of curvature of $\partial\cD$. Thus our analysis may be carried out
provided all the cones and discs we consider are small enough. The conditioning
on the event $G_k$ which we used in Section~\ref{sec:analysis} precisely
guarantees that for all $n$ large enough, on $G_k$, all the regions we consider
are small enough ($O(\log n/\sqrt n\,)=o(1)$ in this scaling), and that $G_k$
still occurs with high probability. These remarks yield the following result.
As before, $\cD_n=\sqrt n \cD$ denotes the scaling of $\cD$ with area $n$.

\begin{proposition}\label{pro:cone-walk}
Let $\cD$ be a fixed smooth convex domain of area $1$ and diameter $\delta$. 
Consider a Poisson point process $\ppp_z^n$ of intensity $1$ contained in 
$\cD_n=\sqrt n \cD$. Let $z,q\in \cD_n$. Let $\ppp_z^n$ be 
$\ppp^n$ conditioned on $z\in \ppp^n$. Then, there exist constants 
$C_{\Gamma,\cD}$, $A_{\Gamma,\cD}$, $\Gamma\in \{\kappa,K,\Lambda,T\}$, 
such that for the cone walk on $\ppp_z^n$, and all $n$ large enough, we have
$$\sup_{z,q\in \cD_n} \p\pth{\Gamma(z,q)> C_{\Gamma,\cD} \cdot \|zq\| 
+ (1+\sqrt{\|zq\|}) \omega_n^5\,} 
\le A_{\Gamma ,\cD} \cdot \expo{-\omega_n^{3/2}}.$$
\end{proposition}

Thus we obtain upper tail bounds for the number of steps $\kappa(z,q)$, the
number of visited sites $K(z,q)$, the length $\Lambda(z,q)$ and the complexity
$T(z,q)$ which are \emph{uniform} in the starting point $z$ and the location of
the query $q$. We now move on to strengthening the results to the worst queries
(still in a random Delaunay).

\subsection{The worst query in a Poisson Delaunay triangulation}
In this section, we prove a Poissonised version of Theorems~\ref{thm:local}
and~\ref{th:main-th}. The proof of the latter are completed in
Section~\ref{sec:depoiss}. Consider $\sup_{z\in \ppp^n,q\in \cD_n}\Gamma(z,q)$,
the value of the parameter for the \emph{worst possible} pair of starting point
and query location, for $\Gamma\in \{\kappa,K,\Lambda, T\}$.

\begin{theorem}\label{thm:main_poissonized}
There exist a constant $C_{\Gamma,\cD}$ depending only on 
$\Gamma$ and on the shape of $\cD$ such that, for all $n$ large enough,
\begin{align*}
\p\bigg(\exists z\in \ppp^n,\, q\in \cD_n~:~\Gamma(z,q) 
>
C_{\Gamma,\cD} \cdot \|zq\| + (1+\sqrt{\|zq\|})\log^4 n\,\bigg)
& \le \frac 1{n^2}.
\end{align*}
\end{theorem}

By Lemma~\ref{lem:invariance}, the number of possible walks in a given Delaunay
tessellation of size $n$ is at most $n\times n^{4}$. Although it \emph{seems}
intuitively clear that this should be sufficient to bound the parameters for
the worst query, it is not the case: one needs to guarantee that there is a way
to sample points \emph{independently} from $\ppp^n$ such that all the cells are
hit, or in other words that the cells are not too small. We prove the following:

\begin{proposition}[Stability of the walk]
\label{pro:area_smallest_cell}There exists a partition of $\cD$ into at most
$2n^4$ cells such that the sequence of steps used by the cone walk algorithm
from any vertex of the triangulation does not change when $q$ moves in a region
of the partition. Furthermore, let $a(\ppp^n)$ denote the area of the smallest
cell. Then $$\p(a(\ppp^n)< \eta) < 1-O(n^8 \eta^{1/3}),$$ for all $n$ large
enough.
\end{proposition}

The proof of Proposition~\ref{pro:area_smallest_cell} relies on bounds on the
smallest angle and on the shortest line segment in the line arrangement. The
argument is rather long, and we think it would dilute too much the focus of
this section, so we prove it in Section~\ref{sec:area}.

\begin{proof}[Proof of Theorem~\ref{thm:main_poissonized}]
In order to control the behaviour of the walk aiming at any point $q\in \cD$,
it suffices to control it for one point of each face of the subdivision
associated with the line arrangement $\Xi(\ppp^n)$ introduced in
Section~\ref{sec:geometry}. Let $\cQ_n=\{q_i: 1\le 1\le n^{2k})\}$, for some
$k\ge 1$ to be chosen later, be a collection of i.i.d.\ uniform random points
in $\cD$, independent of $\ppp^n$. Let $E_n$ be the event that every single
face of the subdivision contains at least one point of $\cQ_n$. Since there are
at most $|\ppp^n|^4$ regions, Proposition~\ref{pro:area_smallest_cell} implies
that
\begin{align}\label{eq:bound_enc}
    \p(E_n^c) 
    & \le \p\big(E_n^c~|~a(\ppp^n)\ge n^{-k} \big) 
        + \p \big( a(\ppp^n)< n^{-k} \big)\notag\\
    & \le \e[|\ppp^n|^4] \cdot \p(\text{Bin}(n^{2k}, n^{-k})=0) + 
        O(n^8 \cdot n^{-k/3}) \notag\\
    & = O(n^4 \cdot \expo{-n^{k/2})} + O(n^{8-k/3}) \notag\\
    & = O(n^{8-k/3}).
\end{align}
For $x>0$ write $f_n(x):=C_{\Gamma,\cD} x + (1+\sqrt x) \log^4 n$. Then, we have
\begin{align}\label{eq:bound_bad_pair}
    & \p\big(\exists z\in \ppp^n, q\in \cD_n~:~ \Gamma(z,q) 
    \ge f_n(\|zq\|)\big)\notag\\
    & \le \p\big(\exists z\in \ppp^n, q\in \cD_n~:~ \Gamma(z,q) 
    \ge f_n(\|zq\|)~\big|~ E_n\big)
    + \p(E_n^c)\notag\\
    & = \p\big(\exists z\in \ppp^n, q\in \cQ_n~:~ \Gamma(z,q) 
    \ge f_n(\|zq\|)\big) 
    + \p(E_n^c)\notag\\
    & \le n^{2k} \cdot 
    \sup_{q\in \cD} \p\big(\exists z\in \ppp^n~:~\Gamma(z,q) 
    \ge f_n(\|zq\|) \big) 
    + \p(E_n^c).
\end{align}
Now for any fixed $q\in \cD$, and conditioning on $\ppp^n$, we see that
\begin{align*}
    \p\big(\exists z\in \ppp^n~:~\Gamma(z,q) \ge f_n(\|zq\|) \big)
    & \le \e \bigg[ \sum_{z\in \ppp^n} \I{\Gamma(z,q) \ge f_n(\|zq\|)}\bigg]\\
    & \le \e\bigg[\sum_{z\in \ppp^n} \I{\Gamma(z,q)\ge f_n(\|zq\|), |\ppp^n|\le 
    2n}\bigg] + \p\pth{|\ppp^n|>2n}\\
    & \le 2 n \sup_{z\in \cD} \p(\Gamma(z,q)\ge f_n(\|zq\|)) + \expo{-n/3}.
\end{align*}
The claim follows from \eqref{eq:bound_enc}, \eqref{eq:bound_bad_pair}, and
Proposition~\ref{pro:cone-walk} by choosing $k=40$. (Note that we could easily
obtain sub-polynomial bounds.)
\end{proof}

Before concluding this section, we note that the same arguments also yield the
Poissonised version of Theorem~\ref{thm:local} about the number of
neighbourhoods required to compute all the steps of any query. We omit the
proof. \begin{theorem}\label{thm:local_poisson} We have the following bound for
any run of $\proc{Cone-walk}$ on $\DT(\ppp^n)$, for $\ppp^n$ a Poisson point
process of rate $n$ in $\cD$:
$$
    \p\Big(\exists z \in \ppp^n, q\in \cD_n: M(z,q) 
    > \log^{1+\epsilon} n \Big) \le \frac 1 {n^2}.
$$
\end{theorem}

\subsection{De-Poissonisation: Proof of Theorem~\ref{th:main-th}}
\label{sec:depoiss}
In this section, we prove that the results proved in the case of a Poisson
point process of intensity $n$ in a compact convex domain $\cD$
(Theorem~\ref{thm:main_poissonized}) may be transferred to the situation where
the collection of sites consists of $n$ independent uniformly random points in
$\cD$. In the present case, the concentration for our events is so strong that
the de-Poissonisation is straightforward.

Note that conditional on $|\ppp^n|=n$, the collection of points $\ppp^n$ is
precisely distributed like $n$ independent uniforms in $\cD$. Furthermore,
$|\ppp^n|$ is a Poisson random variable with mean $n$, so that by Stirling's
formula, as $n\to\infty$,
\begin{equation}\label{eq:llt}
    \p(|\ppp^n|=n)= \frac{n^n \expo{-n}}{n!} \sim \frac 1{\sqrt{2\pi n}}.
\end{equation}
This is small, but we can consider multiple copies of the process; if one
happens to have exactly $n$ points, but that none of them actually behaves
badly, then it must be the case that the copy that has $n$ points behave as it
should, which is precisely what we want to prove. To make this formal, consider
a sequence $\ppp^n_i$, $i\ge 1$, of i.i.d.\ Poisson point processes with mean
$n$ in $\cD$. Let $\eta_n$ be the first $i\ge 1$ for which $|\ppp^n_i|=n$.
Then, for any event $E^i$ defined on $\ppp^n_i$,
\begin{align*}
    \p(E_1\,|\, |\ppp^n_1|=n)
    & = \p(E_{\eta_n}) \\
    & \le \p(\exists j\le \eta_n: E_j)\\
    & \le \p(\exists j\le n: E_j, \eta_n\le n)\\
    & \le n \p(E_1) + \p(\eta_n\ge n)\\
    & \le n \p(E_1) + \expo{-\sqrt{n/(2\pi)}},
\end{align*}
where the last line follows from \eqref{eq:llt}. In the present case, we apply
the argument above to estimate the probabilities that the number of steps, the
number of sites visited, the length of the walk or the complexity of the
algorithm exceeds a certain value. In the case of the Poisson point process, we
have proved that any of these events has probability at most $1/n^2$, so that
we immediately obtain a bound of $O(1/n)$ in the case of $n$ i.i.d.\ uniformly
random points, which proves the main statement in Theorem~\ref{th:main-th}.

This implies immediately that, for $\delta$ the diameter of $\cD$, we have
\begin{align*}
\e\bigg[\sup_{z\in \bX^n,q\in \cD_n} \Gamma(z,q)\bigg] 
    & \le 2\delta C_\Gamma \sqrt n + 
    \p\bigg(\sup_{z\in \bX_n,q\in \cD_n} 
        \Gamma(z,q) \ge 2\delta C_\Gamma \sqrt n\bigg)\\
    & \le 2 \delta C_\Gamma \sqrt n + O(1/n)\\
    & \le 3 \delta C_\Gamma \sqrt n,
\end{align*}
for $n$ large enough, so that the proof of Theorem~\ref{th:main-th} is
complete. The proof of Theorem~\ref{thm:local} from
Theorem~\ref{thm:local_poisson} relies on the same ideas and we omit the proof.

\section{Extremes in the arrangement $\Xi(\ppp^n)$} \label{sec:area}
In this section, we obtain the necessary results about the geometry of the line
arrangement $\Xi$ introduced in Section~\ref{sec:stability}. In particular, we
prove Proposition~\ref{pro:area_smallest_cell} providing a uniform lower bound
on the areas of the cells of the subdivision associated to $\Xi$ which is
crucial to the proofs of our main results Theorems~\ref{thm:local}
and~\ref{th:main-th}.

The rays (half-lines) of the arrangement $\Xi$ all have one end at a site of
$\ppp^n$. Let $\Xi_x$ denotes the set of rays of $\Xi$ with one end at $x$; we
say that the rays in $\Xi_x$ originate at $x$. The rays come in two kinds:
\begin{itemize}
    \item \emph{regular rays} are defined by a triple $x,y,y'$ such that there
    exists $r>0$ and $q\in \bbR^2$ for which $y$ and $y'$ both lie on the front
    arc of the $\cone(x,q,r)$; such a ray is denoted by $\rho^-(x,y,y')$;
    \item \emph{extreme rays} correspond to one of the straight boundaries of
    $\cone(x,y,\infty)$, for some pair of points $x$, $y$; these two rays are
    denoted by $\rho^+_1(x,y)$ and $\rho^+_2(x,y)$ in such a way that the angle
    from $\rho^+_1$ to $\rho^+_2$ is $\pi/4$.
\end{itemize}
We let $\Xi_x^-$ and $\Xi_x^+$ denote the collections of regular and extreme
rays of $\Xi_x$, respectively. Our strategy to bound the area of the smallest
cell is to lower bound the angle between any two rays (Section~\ref{sec:angle})
and the length of any line segment (Section~\ref{sec:length}). We put together
the arguments and prove Proposition~\ref{pro:area_smallest_cell} in
Section~\ref{sec:proof_area}.

\subsection{The smallest angle between two rays}\label{sec:angle}
We start with a bound on the smallest angle in the arrangement. 

\begin{lemma}\label{lem:smallest_angle}
Let $\langle \ppp^n \rangle$ denote the minimum 
angle between any two lines of $\Xi(\ppp^n)$ which intersect within $\cD$. 
Then, there exists a constant $C$ such that, for any $\beta\in (0,\pi/8)$, 
$$\p(\langle \ppp^n \rangle < \beta)\le C n^5\beta.$$
\end{lemma}
\begin{proof}
For two intersecting rays $\rho_1$ and $\rho_2$, let $\langle
\rho_1,\rho_2\rangle$ denote the (smallest) angle they define. We first deal
with the angles between two rays originating from different points. Consider a
given ray $\rho_1\in \Xi_{x_1}$, and let $\tilde x_1$ denote the intersection
of $\rho_1$ with $\partial \cD$. For another ray $\rho_2\in \Xi_{x_2}$, with
$x_2\neq x_1$, to intersect $\rho_1$ at an angle lying in $(-\beta,\beta)$, we
must have $x_2 \in \cone_\beta(x_1,\tilde x_1, \infty) \cup \cone_\beta(\tilde
x_1,x_1,\infty)$. Observe that for any two points $x,y$, the area of
$\cone_\beta(x,y,\infty)\cap \cD$ is bounded by $\delta^2 \tan \beta$, where
$\delta$ denotes the diameter of $\cD$. In particular, for any fixed ray
$\rho_1 \in \Xi_x$ for some $x\in \ppp^n$,
\begin{align*}
    & \p(\exists x_2\in \ppp^n_x\setminus \{x\}, \rho_2 \in \Xi_{x_2}: 
        \langle \rho_1,\rho_2\rangle< \beta) \\
    & \le \p(\exists x_2\in \ppp^n_x\setminus\{x\}: 
        x_2 \in \cone_\beta(x_1,\tilde x_1, \infty)) \\
    & \quad + \p(\exists x_2\in \ppp^n_x\setminus\{x\}: 
        x_2 \in \cone_\beta(\tilde x_1,x_1,\infty))\\
    & \le 2 n \delta^2 \tan \beta,
\end{align*}
since $\ppp_x^n\setminus \{x\}$ is a Poisson point process in $\cD$ with
intensity $n$. Now, since for any $x$, $|\Xi_x|\le 2 (|\ppp^n|-1)$,
\begin{align}\label{eq:angle_inter_diff}
    & \p(\exists x_1,x_2\in \ppp^n, x_1\neq x_2, 
        \rho_1\in \Xi_{x_1}, \rho_2\in \Xi_{x_2}: 
    \langle \rho_1,\rho_2 \rangle<\beta ) \notag\\
    & \le \e\Bigg[ 
    \sum_{x\in \ppp^n} 
    \p(\exists \rho_1 \in \Xi_{x}, x_2\in \ppp^n_x 
        \setminus \{x\}, \rho_2 \in \Xi_{x_2}: 
    \langle \rho_1,\rho_2\rangle< \beta) \Bigg] \notag \\
    & \le \e \Bigg[ \sum_{x\in \ppp^n} 2 (|\ppp^n|-1) 
        \cdot 2 n \delta^2 \tan \beta \Bigg] \notag \\
    & = 4 n^3 \delta^2 \tan \beta,
\end{align}
since $\ppp^n_{x}\setminus \{x\}$ is distributed like $\ppp^n$. 

We now deal with the smallest angle between any two rays in the same set
$\Xi_x$, for some $x\in \ppp^n$. Any point $y\in \ppp^n_x\setminus\{x\}$
defines at most two extreme rays in $\Xi_x$ which intersect at an angle of
$\pi/4$. For any two distinct points $y,y'$ to define two extreme rays that
intersect at an angle smaller than $\beta>0$, then $y'$ must lie in
$\cone_\beta(x,y,\infty)$ or one of its two images in the rotations of angle
$+\pi/8$ or $-\pi/8$ about $x$. The arguments we have used to obtain
\eqref{eq:angle_inter_diff} then imply that
\begin{align*}
    \p(\exists x\in \ppp^n:\exists \rho,\rho'\in \Xi_x^+, 
        \langle \rho, \rho'\rangle < \beta)
    & \le 3 n^3 \delta^2 \tan \beta\\
    & \le 6 n^3 \delta^2 \beta,
\end{align*}
for all $\beta\in [0,\pi/4]$, since then $\tan \beta \le 2\beta$.

We now consider a regular ray $\rho=\rho^-(x,y,z)$. Then there exist $q\in
\bbR^2$ and $r>0$ such that $\rho$ is the half-line with an end at $x$ and
going through $q$, and $x,y,z$ all lie on the same circle $\sC$ (more
precisely, $y,z$ lie on the intersection of $\sC$ with $\cone(x,q,\infty)$). We
first bound the probability that there exists a point $z'$ such that the
regular ray $\rho'=\rho^-(x,y,z')$ intersects $\rho$ at an angle at most
$\beta$, that is $\rho'\in \cone_\beta(x,q,\infty)$. We want to bound the area
of the region
$$\{z'\in \cD: \rho^-(x,y,z') \in \cone_\beta(x,q,\infty)\},$$
for which the angle between $\rho$ and $\rho'$ is at most $\beta$. By
definition of $\rho'$, $x,y$ and $z'$ lie on the same circle $\sC'=\sC'(c)$,
which has a centre $c'$ in $\cone_\beta(x,q,\infty)$. More precisely, the
centre $c'$ lies on the intersection of the bisector of the line segment
$[x,y]$ with $\cone_\beta(x,q,\infty)$; so $c'$ lies in a line segment $\ell$
of length $O(\beta)$ containing the centre $c$ of $\sC$. The Hausdorff distance
between $\sC$ and $\cup_{c'\in \ell} \sC'(c)$ is $O(\beta)$ and it follows that
$$\cA \{z'\in \cD: \rho^-(x,y,z') \in \cone_\beta(x,q,\infty)\} = O(\beta).$$
A similar argument applies to deal with rays $\rho^-(x,y',z')$ for $y',z'\in
\ppp^n$: once the one of $y'$ or $z'$ is chosen, the second is forced to lie in
a region of area $O(\beta)$. From there, arguments similar to the ones we used
in \eqref{eq:angle_inter_diff} yield
\begin{equation}\label{eq:angle_ext-ext}
    \p(\exists x \in \ppp^n: \exists \rho,\rho'\in \Xi_x^-, 
    \langle \rho, \rho' \rangle < \beta) = O(n^5 \beta).
\end{equation}

Finally, we deal with the intersections of a regular and an extreme ray. We
$\rho=\rho^-(x,y,z)$ and $\rho'=\rho^+_1(x,y')$, for $y'\in \ppp^n$ (potentiall
$y'=y$). For $\rho'$ to lie in $\cone_\beta(x,q,\infty)$, the point $y'$ must
lie in a cone with apex $x$, half-angle $\beta$ and with axis one of the
straight boundaries of $\cone(x,q,\infty)$. So if $y\ne y'$, the point $y'$ is
constrained to lie in a region of area $O(\beta)$; in this case, the expected
number of $x,y,z,y'\in \ppp^n$ in such a configuration is $O(n^4 \beta)$. If on
the other hand, $y'=y$ then $y$ must lie on a portion of length $O(\beta)$ of
the front arc of $\sC \cap \cone(x,q,\infty)$. However, conditional on $x,y,z$
indeed forming the ray $\rho=\rho^-(x,y,z)$, the points $y$ is uniformly
distributed on the front arc. In other words, the expected number of triplets
$x,y,z$ in such a configuration is $O(n^3 \beta)$. It follows that
\begin{equation}\label{eq:angle_ext-reg}
    \p(\exists x\in \ppp^n: \exists \rho \in \Xi_x^-, \rho'\in \cS_x^+, 
\langle \rho,\rho' \rangle < \beta) = O(n^4 \beta).
\end{equation}
Putting \eqref{eq:angle_inter_diff}, \eqref{eq:angle_ext-ext} and
\eqref{eq:angle_ext-reg} together completes the proof.
\end{proof}

\subsection{The shortest line segment}
\label{sec:length}

Let $L(\ppp^n)$ denote the smallest line segment in the arrangement of lines
$\Xi(\ppp^n)$. The aim of this section is to prove the following:

\begin{lemma}\label{lem:shortest_segment}
We have, for any $\gamma$ small enough and $n$ large enough,
$$\p(L(\ppp^n)<\gamma ) \le n^9 \gamma^{1/2}.$$
\end{lemma}

In order to prove Lemma~\ref{lem:shortest_segment}, we place a ball of radius
$\gamma$ at every intersection in $\Xi(\ppp^n)$ (including the ones which
involve the boundary $\partial \cD$) and estimate the probability that any such
ball is hit by another ray of the arrangement. The proof is longer and slightly
more delicate than that of Lemma~\ref{lem:smallest_angle} simply because three
rays, each involving up to three points of $\ppp^n$, might be involved the
event of interest. The first step towards proving
Lemma~\ref{lem:shortest_segment} is to establish uniform bounds on the area of
location of some points for the rays to intersect a given ball. These bounds
are stated in Lemmas~\ref{lem:loc_points_ray-ext}
and~\ref{lem:loc_points_ray-reg} below.

\begin{lemma}\label{lem:loc_points_ray-ext}
There exists a constant $C$ such that, for any $x,y,z\in \cD$, one has,
for all $\gamma>0$ small enough, we have
\begin{compactenum}[(i)]
    \item $\cA \{\bar x\in \cD: \rho^+(\bar x,y)
          \cap \delta(z,\gamma) \ne \varnothing\} \le C \gamma$
    \item $\cA \{\bar y\in \cD: \rho^+(x,\bar y) 
          \cap \delta(z,\gamma) \ne \varnothing\} \le C \gamma/ \|xz\|.$
\end{compactenum}
\end{lemma}
\begin{proof}
Fix $x$, $z$ and $\gamma>0$. The location of points $\bar y$ for which
$\rho^+(x,\bar y)\ni z$ is precisely $\partial \cone(x,z,\infty)$. Then,
$$\bigcup_{q\in \delta(z,\gamma)} \partial \cone(x,q,\infty)$$ consists of two
cones of half-angle $\arcsin(\gamma/\|xz\|)$, which proves the second claim.

For the first claim, let $c_1=c_1(z)$ and $c_2=c_2(z)$ be the two intersections
of the circles of radius $\|zy\|/\sqrt 2$ centred at $y$ and $z$ (so the
segment $zy$ is seen from $c_1$ and $c_2$ at an angle of $\pi/2$). Let
$\sC_1=\sC_1(z)$ and $\sC_2=\sC_2(z)$ denote the discs centred at $c_1$ and
$c_2$, respectively, and having $y$ and $z$ on their boundaries. Then, the
region of the points $\bar x$ for which $\rho^+(\bar x,y)\ni z$ is precisely
the boundary of $\sC_1 \cup \sC_2$. When looking for the region of points $\bar
x$ satisfying $\rho^+(\bar x,y) \cap \delta(z,\gamma)\ni q$, for some
$q\in \delta(z,\gamma)$ observe that the centres $\|c_1(q)-c_1(z)\|=O(\gamma)$, 
$\|c_2(q)-c_2(z)\|=O(\gamma)$ and that the Hausdorff distance between 
$\sC_1(q)\cup \sC_2(q)$ and $\sC_1\cup \sC_2$ is $O(\gamma)$, hence the first claim.
\end{proof}

\begin{lemma}\label{lem:loc_points_ray-reg}
There exists a constant $C$ such that, for any $x,y,y',z \in \cD$, 
one has, for all $\gamma>0$ small enough
\begin{compactenum}[(i)]
    \item $\cA \{\bar x\in \cD: \rho^-(\bar x,y,y') \cap \delta(z,\gamma) 
    \ne \varnothing\} \le C \gamma $
    \item $\cA \{\bar y\in \cD: \rho^-(x,\bar y, y') \cap \delta(z,\gamma) 
          \ne \varnothing\} \le C \gamma.$
\end{compactenum}
\end{lemma}
\begin{proof}
The ray $\rho^-(x,y,y')$ is the half-line with one end at $x$ and going through
the centre $c$ of the circle which contains all three $x,y$ and $y'$. So the
region of the points $\bar x$ for which $\rho^-(\bar x, y,y')\ni z$ is
naturally indexed by the points of the bisector $\Delta$ of the segment
$[y,y']$: for a point $c\in \Delta$, the only possible points $\bar x=\bar
x(c)$ lie the intersection of the circle $\sC$ centred at $c$ and containing
$y$ (and $y'$) and the line containing $c$ and $z$. So if $z\ne c$, there are
two such points and at most one of them is such that $y,y'\in \cone(\bar x,
c,\infty)$. If $z=c$, then there is potentially a continuum of points $\bar
x(c)$, consisting of the arc of $\sC$ for which $y,y'\in \cone(\bar
x,c,\infty)$. It follows that for every $c\in \Delta$ such that
$\|zc\|>2\gamma$, the points $\bar x(c)$ for which $\rho^-(\bar x,y,y')$
intersects $\delta(z,\gamma)$ is an arc of length $O(\gamma)$. On the other
hand, the cumulated area of the circles $\sC(c)$ for which $\|zc\|\le 2\gamma$
is $O(\gamma)$. Hence the first claim.

For $x$, $y$ and $q$ fixed, the set of points $\bar y$ such that $\rho^-(x,y,
\bar y)\ni q$ is contained in the circle $\sC'(q)$ whose centre is the
intersection of the bisector of $[x,y]$ and the line $(xq)$, and which contains
$x$ (and $y$). The cumulated area of $\sC'(q)$ when $q\in \delta(z,\gamma)$ is
$O(\gamma)$, which proves the second claim.
\end{proof}

\begin{proof}[Proof of Lemma~\ref{lem:shortest_segment}]
Recall that, in order to lower bound the length of the shortest line segment
in $\Xi(\ppp^n)$, we prove using Lemmas~\ref{lem:loc_points_ray-ext}
and~\ref{lem:loc_points_ray-reg}, that if we place a ball of radius $\gamma$ at
every intersection in $\Xi(\ppp^n)$, every ball contains a single intersection
with probability at least $1-O(n^9 \gamma^{1/2})$. The intersections to
consider are of three different types: the points $x\in \ppp^n$, the
intersections of one ray and the boundary of the domain $\cD$, and
intersections between two rays.

The case of balls centred at points of $\ppp^n$ first is easily treated. We
have
\begin{align*}
    \p(\exists z,x\in \ppp^n: z\neq x, \Xi_x\cap \delta(z,\gamma) 
    \neq \varnothing )
    & \le \e \left[
    \sum_{z\in \ppp^n} \I{\exists x\in \ppp^n \setminus\{z\}: 
    \Xi_x \cap \delta(z,\gamma)\ne \varnothing}\right].
\end{align*}
Note that, in the expected value on the right, although $x\ne z$, the set
$\Xi_x$ does depend on $z$ in a non-trivial way. The main point about the proof
is to check that the conditioning on $z\in \ppp^n$ does not make it much easier
for a ray to intersect $\delta(z,\gamma)$; and this is where
Lemmas~\ref{lem:loc_points_ray-ext} and~\ref{lem:loc_points_ray-reg} enter the
game. We have
\begin{align*}
    \e \left[
    \sum_{z\in \ppp^n} \I{\exists x\in \ppp^n \setminus\{z\}: \Xi_x \cap 
    \delta(z,\gamma)\ne \varnothing}\right]
    & \le \e \left[ 
    \sum_{z\in \ppp^n} \sum_{y\in \ppp^n} \sum_{x\in \ppp^n\setminus \{z,y\}} 
    \I{\rho^+(x,y)\cap \delta(z,\gamma)\ne \varnothing}
    \right] \\
    & + \e \left[ 
    \sum_{z\in \ppp^n} \sum_{y,y'\in \ppp^n} 
    \sum_{x\in \ppp^n\setminus \{z,y,y'\}} 
    \I{\rho^-(x,y,y')\cap \delta(z,\gamma)\ne \varnothing}
    \right]\\
    & \le C n^4 \gamma,
\end{align*}
by Lemmas~\ref{lem:loc_points_ray-ext} and~\ref{lem:loc_points_ray-reg}, since
conditional on $z,y,y'\in \ppp^n$, $\ppp^n\setminus \{x,y,y'\}$ is distributed
like $\ppp^n$.

We now move on to the intersections of rays with the boundary of $\cD$. Let
$\rho$ be a ray, and let $z$ denote its intersection with $\partial \cD$. If
$\rho$ is an extreme ray, then there exists $x,y\in \ppp^n$ such that
$\rho=\rho^+_1(x,y)$ (or $\rho=\rho^+_2(x,y)$) and we are interested in the
probability that some other ray $\rho'$ intersects $\delta(z,\gamma)$; note
that the ray $\rho'$ might arise from a set of points $x',y'$ (if it is
extreme) or $x',y',y''$ (if it is regular) which uses some of $x$ or $y$. If
there is at least one point of $x',y'$ that is not in $\{x,y\}$ then
Lemma~\ref{lem:loc_points_ray-ext} allows us to bound the probability that
$\rho^+(x',y')$ intersects $\delta(z,\gamma)$; in the other case, there must be
one of $x',y',y''$ which is not in $\{x,y\}$ and
Lemma~\ref{lem:loc_points_ray-reg} applies: For each such choice of points, the
probability that $\rho'$ intersects $\delta(z,\gamma)$ is $O(\gamma)$,
uniformly in $x$, $y$ and $z$. So the only case remaining to check is when
$\rho'$ is only defined by the two points $x,y$. Here,
\begin{itemize}
    \item either $\rho'\in \Xi_x$ and for $\rho'$ to intersect
    $\delta(z,\gamma)$ one must have $x\in \delta(z,\gamma)$, which only
    happens if $x$ lies within distance $\gamma$ of $\partial \cD$, hence with
    probability at most $O(n\gamma)$; 
    \item or $\rho'\in \Xi_y$ and the ball
    $\delta(z,\gamma)$ should be close to one of the two points in $\partial
    \cone(x,y,\infty)\cap \partial \cone(y,x,\infty)$; Write $i_1=i_1(x,y)$ and
    $i_2=i_2(x,y)$ for these two points. More precisely, since the angle
    between $\rho$ and $\rho'$ is $\pi/4$, the point $z$ should lie within
    distance $\gamma \sqrt 2$ of one of the two intersections (or
    $\delta(z,\gamma)$ does not intersect $\rho'$). Dually, for a point $q$,
    let $\bar y=\bar y(q)$ be the point such that $i_1(x,\bar y)=q$. Then, for
    $\|zq\|\le \gamma'=\gamma \sqrt 2$ the angle between the vectors $x\bar
    y(q)$ and $x\bar y(z)$ is at most $\arcsin(\gamma'/\|xz\|)$ and the
    difference in length is an additive term of at most $\gamma' \sqrt 2$.
    Overall, $\{\bar y: i_1(x,\bar y) \in \delta(z,\gamma')\}$ is contained in
    a ball of radius at most $3\gamma' \sqrt 2=6\gamma$. Finally, since $z\in
    \partial \cD$, the probability that such a situation occurs for some
    $x,y\in \ppp^n$ is at most $O(n^2 \gamma)$.
\end{itemize}
As a consequence, the probability that there exist two rays $\rho$ and $\rho'$,
such that $\rho'\cap \delta(z,\gamma)\ne \varnothing$, for $z=\rho\cap \cD$ is
$O(n^5 \gamma)$.

It now remains to deal with the case of balls centred at the intersection of
two rays. Consider two rays, $\rho_1\in \Xi_{x_1}$ and $\rho_2\in \Xi_{x_2}$,
for $x_1\ne x_2$, which are supposed to intersect at a point $z$ (the case
$x_1=x_2$ has already been covered since then, we have $z\in \ppp^n$). The
definition of $\rho_1$ and $\rho_2$ may involve up to six points $x_1,y_1,y_1'$
and $x_2,y_2,y_2'$ of $\ppp^n$. The probability that some third ray $\rho_3$,
whose definition uses at least one new point of $\ppp^n$, intersects
$\delta(z,\gamma)$ can be upper bounded using
Lemmas~\ref{lem:loc_points_ray-ext} and~\ref{lem:loc_points_ray-reg} as before
and we omit the details. The only new cases we need to cover are the ones when
the definition of $\rho_3$ only involves points among $x_1,y_1,y_1'$ and
$x_2,y_2,y_2'$. We now show that all the possible configurations on six points
only are unlikely to occur for $\gamma$ small in the random point set $\ppp^n$.

Because of the number of cases to be treated, without a clear way to develop a
big picture, we only sketch the remainder of the proof. We treat the cases
where $\rho_1=\rho^-(x_1,y_1,y_1')$ and $\rho_2=\rho^-(x_2, y_2, y_2')$, so in
particular the two rays which intersect at $z$ are regular rays.
\begin{itemize}
    \item If $\rho_3\in \Xi_{x_1}$ (or, by symmetry, $\rho_3\in \Xi_{x_2}$),
    the bound on the smallest angle between any two rays in
    Lemma~\ref{lem:smallest_angle} ensures that $z$ ought to be close to $x_1$:
    one must have $\arcsin(\gamma/\|x_1z\|)\le \eta$ where $\eta$ denotes the
    smallest angle. Then, either $x_2$ is also close to $z$, or $y_2$ and
    $y_2'$ actually almost lie on a circle whose centre lies on the line
    $(x_1,x_2)$. Both are unlikely: (1) the closest pair of points lie at
    distance $\Omega(1/n)$ (which is much larger than what we are aiming for)
    and (2) for fixed $x_1,x_2$ and $y_2$, the location of points $y_2'$ such
    that $y_2$ lies near the circle whose centre lies on $(x_1,x_2)$ and which
    contains $y_2$ has area of order $O(\gamma)$.
    \item If $\rho_3\in \Xi_{y_1}$ (or the other symmetric cases), there are
    the following possibilities: either $\rho_3$ is a regular ray and (up to
    symmetry)
    $$
        \rho_3
        \in 
        \left\{
        \begin{array}{l l l l}
        \text{(a)}& \rho^-(y_1,x_1,y_1'), & \text{(b)}& \rho^-(y_1,x_1,y_2)\\
        \text{(c)}&\rho^-(y_1, x_2, y_1'), & \text{(d)}&\rho^-(y_1,x_2, y_2)\\
        \text{(e)}&\rho^-(y_1,y_1',y_2'), & \text{(f)}&\rho^-(y_1,y_2,y_2')
        \end{array}
        \right\},
    $$ 
    or $\rho_3$ is an extreme ray and 
    $$
        \rho_3 \in 
        \big\{
        \begin{array}{l l l l}
        \text{(g)}& \rho^+(y_1,y_1'), & \text{(h)} & \rho^+(y_1,y_2)
        \end{array}
        \big\}.
    $$
    Since $\rho^-(x_1,y_1,y_1')$ is a ray, $y_1,y_1'$ both lie in a cone of
    half-angle $\pi/8$, and it is impossible that we also have $x_1,y_1'$ both
    lying in some cone of half-angle $\pi/8$; so configuration (a) does not
    occur. For situation (b), we use Lemma~\ref{lem:loc_points_ray-reg} (ii).
    For situations (c)--(f), we use the arguments in the proof of
    Lemma~\ref{lem:loc_points_ray-reg} to exhibit the constraints about the
    location of the third point defining $\rho_3$, once the first two are
    chosen.

    To deal with the cases where $\rho_3\in \Xi_{y_1}^+$, observe that once
    $y_1$ is placed, $y_1'$ (or $y_2$) must lie in a cone of angle
    $\arcsin(\gamma/\|y_1z\|)$: so either $\|y_1z\|\ge r$, and the cone has
    area $O(\gamma/r)$ or $\|y_1z\|\le r$, but then $y_1$ must lie in the
    portion of the front arc of length $O(r)$ about $\rho_1$. Since the
    conditional distribution of $y_1$ is uniform on the arc, the probability
    that this configuration occurs is $O(\inf_{r\ge 0}\{\gamma/r +
    r\})=O(\gamma^{1/2})$.
\end{itemize}
The remaining cases, where at least one of $\rho_1$ and $\rho_2$ is an extreme
ray, may all be treated using similar arguments and we omit the tedious details.
\end{proof}

\subsection{
    Lower bounding the area of the smallest cell: 
    Proof of Proposition~\ref{pro:area_smallest_cell}
} 
\label{sec:proof_area}
With Lemmas~\ref{lem:smallest_angle} and~\ref{lem:shortest_segment} under our
belt, we are now ready to prove Proposition~\ref{pro:area_smallest_cell}.

We use the line arrangement $\Xi(\ppp^n)$ of Section~\ref{sec:geometry}. The
faces of the corresponding subdivision are convex regions delimited by the line
segments between intersections of rays, and pieces of the boundary of the
domain $\cD$. Let $L(\ppp^n)$ and $\langle \ppp^n \rangle$ denote the length of
the shortest line segment and the smallest angle in the arrangement arising from
$\ppp^n$. Consider a face of the subdivision, and one of its vertices $u$ which
does not lie on the $\partial \cD$. Then the triangle formed by the $u$ and its
two adjacent vertices on the boundary of the face is contained in the face. The
lengths of the two line segments adjacent to $u$ are at least $L(\ppp^n)$ long,
and the angle they make lies between $\langle \ppp^n\rangle$ and $\pi-\langle
\ppp^n\rangle$, so that the area of the corresponding triangle is at least
$$L(\ppp^n)^2 \cdot \frac{\tan \langle \ppp^n\rangle} 2.$$ It follows easily
from Lemmas~\ref{lem:smallest_angle} and~\ref{lem:shortest_segment} that the
area of the smallest face it at least $\beta \gamma/2$ with probability at
least $1-O(n^4 \beta) - O(n^9 \gamma^{1/2})$. Choosing $\beta,\gamma$ such that
$n^4\beta = n^9 \gamma^{1/2}$ yields the claim.

\section{Comparison with Simulations}\label{sec:simulations}
We implemented \proc{Cone-Walk} in \texttt{C++} using the CGAL libraries.
For simulation purposes, we generated $10^7$ points uniformly at random in a
disc of area $10^7$. We then simulated $\proc{Cone-Walk}$ on $10^6$ different
walks starting from a point within a disc having a quarter of the radius of the
outer disc (to help reduce border effects) with a uniformly random destination.
For comparison, we give give the expected bounds for a walk whose destination
is at infinity. See Table~\ref{table:comparisons}. With respect to the path, we
give two values for the number of extra vertices visited in a step. The first
is the number obtained by using $\proc{Simple-Path}$ and the second, in
brackets, is the average number of `intermediary vertices' within each disc.

\begin{table}[h!] 
    
   
 \begin{center}
    \small
    \setlength{\tabcolsep}{2pt}
    \renewcommand{\arraystretch}{1.4}   
    \begin{tabular}{| l |  c | c | c |}
            \hline          
            & \textbf{Theory}
            & \textbf{Theory (5 s.f.)}
            & \textbf{Simulation}
        \\ \hline
            Radius
            & $\e[R]=\sqrt{\frac{\pi}{2\sqrt{2} + \pi}}$
            & 0.72542
            & 0.72557
        \\
            \# Intermediary path steps
            & $\leq \e[\tau_i] \leq \frac{4 \pi}{\pi + 2\sqrt{2}}$
            & $\leq$ 1.1049
            & 0.41244 (1.09574)
        \\
            $\proc{Simple-Path}$ Length
            & $
                \leq \e[\Lambda/L_0]\leq
                \frac{22 \pi - 4 \sqrt 2}{2+3\pi+8\sqrt 2}
              $
            & $\le 2.7907$
            & 1.51877
        \\
        \hline
    \end{tabular}
  \end{center}

    \caption{Comparison of theory with simulations. Inequalities are used
             to show when values are bounds. \label{table:comparisons}}
\end{table}

\section*{Appendices}
\appendix

\section{Maximum degree for Poisson Delaunay in a smooth convex}
\label{appendix:max_degree}
Let $\mathbf{\Gamma}$ be a homogeneous Poisson process on the entire Euclidean
plane with intensity 1. In this case,
\citet{BernEY91} give a proof that the expected maximum degree of any vertex of
the Delaunay triangulation $\DT(\mathbf{\Gamma})$ falling within the box
$[0,\sqrt{n}\,]^2$ is $\Theta\left(\frac{\log n}{\log \log n}\right)$. Whilst
such a bound may be useful in the analysis of geometric algorithms, it has a
shortcoming in that it implicitly avoids dealing with the border effects that
occur when considering points distributed within bounded regions. When
considering such bounded regions, it can be observed that the degree
distribution is significantly skewed near the border, with the majority of the
vertices on or near the convex hull having a much higher degree than the global
average. It is therefore not altogether trivial that the maximum degree should
still be bounded polylogarithmicly when the sites are generated by a homogeneous
Poisson process in a bounded region. In this section, we show that this indeed
the case for the specific case of a smooth compact convex. As far as we are
aware, this is the first such bound that has been given for a set of random
points in a bounded region. Let $\cD$ be a smooth compact convex subset of
$\reals^2$ with area $\lambda_2(\cD) = 1$, and $\cD_n=\sqrt{n}\cD$ (with area
$n$ and  diameter$(\cD_n)\leq c_2\sqrt{n}$, for some constant $c_2$). Let $\ppp$
be a homogeneous Poisson point process of intensity 1 contained within $\cD_n$,
so that in expectation we have $\e\,|\ppp|=n$. We use $\delta_\ppp(x)$ to denote
the degree of $x\in\ppp$ in $\DT(\ppp)$, and take $\Delta_\ppp := {\displaystyle
\max_{x\in\ppp}\;\delta_\ppp(x)}$.
\begin{proposition} \label{prop:max_degree}      
    For any $\xi > 0$, we have for $n$ sufficiently large,
    $$
        \p \bigg( \Delta_\ppp \geq \log^{2+\xi} n \bigg) 
        \leq
        \expo{-\log^{1+\xi/4} n}.
    $$
\end{proposition}
Define $\|Ax\| := \inf \left\{ \|xy\| : y \in A \right\}$. Our proof will follow
by considering two cases. For the first case we consider all points $x\in\ppp$
satisfying $\|\partial\cD_n x\| \leq \sqrt{\log n}$ and bound the number of
neighbours of $x$ in $\DT(\ppp)$ to one side of $x$; doubling the result
and the end for the final bound.
To begin, we trace a ray from $x$ to the point $y\in\partial\cD_n$ minimising
$\|xy\|$; we refer to this as $\cR_0$. Next, we create a new ray, $\cR_1$
exiting $x$ such that the area enclosed
by $\cR_0$, $\cR_1$ and $\partial \cD_n$
is $\log^{1+\xi} n$; we let $S_0$ denote this region.
(See Figure~\ref{fig:degree}.) 
For $n$ large enough, the
angle between  the rays $\cR_0$ and $\cR_1$ {\red is} smaller than
$\frac{\pi}{2}+\frac{\pi}{12}$. 
In addition, the length of $\cR_1$ is upper
bounded by the diameter of $\cD_n\leq c_2\sqrt{n}$ and lower bounded by
$c_1\sqrt{\log n}$ for $c_1=\frac{1}{2}$ since $\cR_0$ is not longer than
$\sqrt{\log n}$ by hypothesis.

\begin{figure}[h]
    \begin{center}
        \includegraphics[scale=0.75, page=1]{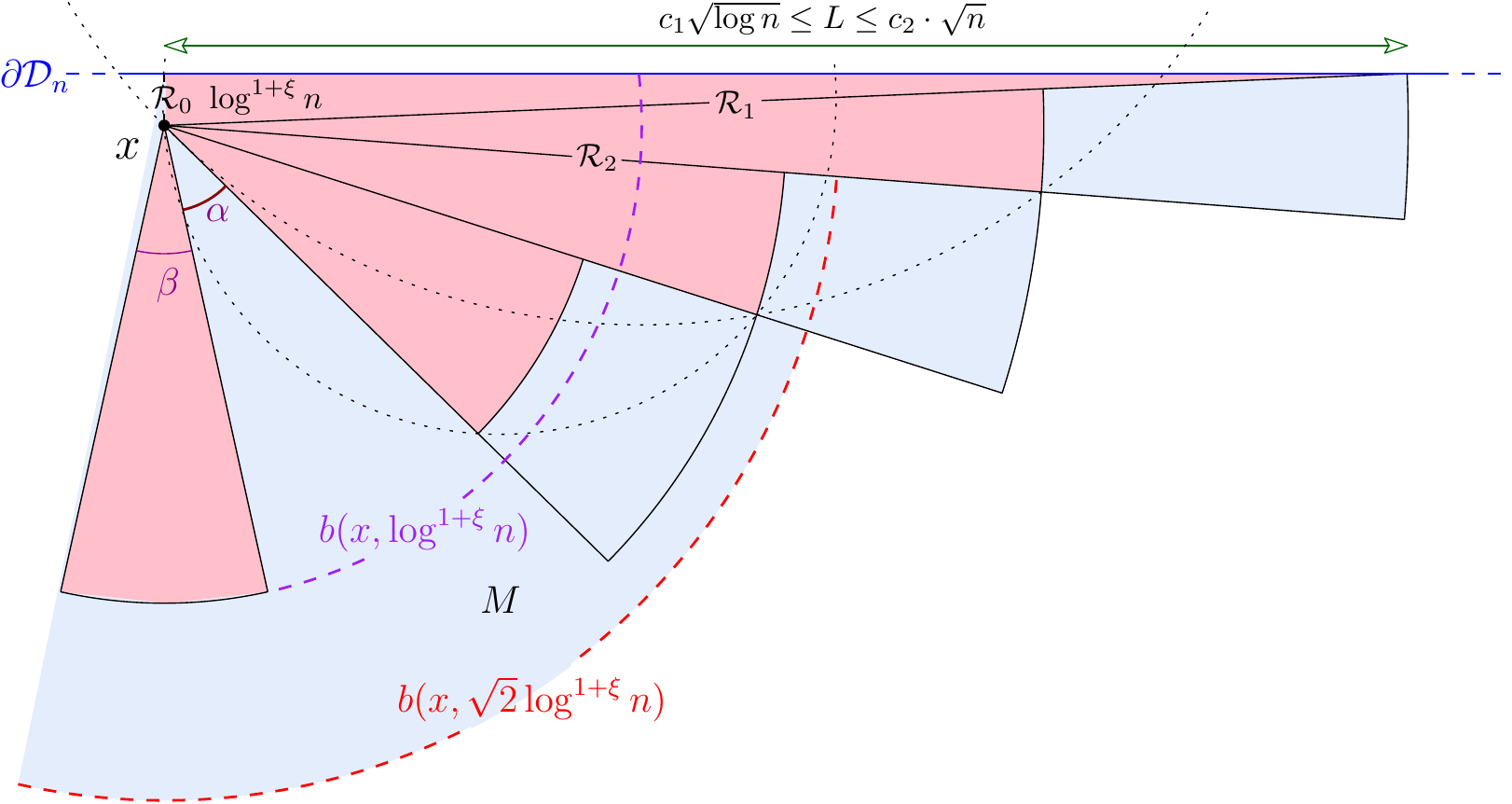}
    \end{center}
    \caption{ 
      \label{fig:degree} The construction to bound the maximum degree of 
      a vertex near the border. The pink shaded circular sectors are each
      conditioned to contain at least one point. In this case, no Delaunay
      neighbour to the right of $x$ can lie outside of the 
    }
\end{figure}

We then define an iterative process that finishes as soon as one of our
regions (that we define shortly) is totally contained within a ball of radius
$\log^{1+\xi} n$ about the point $x$. At each step $i$, we look at the ray
$\cR_{i-1}$ and then grow a sector of a circle with area $\log^{1+\xi}n$, and of
radius $\tfrac{1}{\sqrt{2}} |\cR_{i-1}|$, where $|\cR_i|$ is the length of the
$i$\textit{th} ray. The sector of radius $|\cR_{i+1}|$ defined by the rays 
$\cR_i$ and $\cR_{i+1}$ is denoted by $S_i$. 
Thus the internal angle of each new sector is exactly twice
that of the previous one. For each sector apart from the first, we add a
`border' (shaded blue in Figure~\ref{fig:degree}) which extends the each sector
to the length of the sector proceeding it: let $Q_i$ be the cone delimited by the 
rays $\cR_i$ and $\cR_{i+1}$, and of radius $|\cR_i|$. 
Let $I$ be the index of the first ray $\cR_i$ for which 
$|\cR_i|\le \sqrt 2 \log^{1+\xi} n$, so that the last of this decreasing sequence 
of sectors is $S_{I-1}$.
Finally, we add a sector $S_{I}$ directly
opposite $\cR_0$ within a ball of radius $\log^{1+\xi} n$. We choose its
internal angle $\beta$ so that its area is $\log^{1+\xi} n$. 

We now proceed to showing that each circular sector enclosed by two
adjacent rays contains a point with high probability and that when this event
occurs, we have a bound on the region of points that may be a neighbour to $x$ 
in $\DT(\ppp)$. 

\begin{lemma} \label{lem:alpha_lower}
    Let $\alpha$ be the angle between the final ray $\cR_I$ and the edge of the 
    sector $S_{I+1}$ opposite $\cR_0$. Then, for $n$ large enough, 
    $\alpha$ is positive. This implies that the sectors $S_{i}$, $0\le i \le I$
    are disjoint.
\end{lemma}
\begin{proof}
No angle between any two rays may exceed $2\beta=4\pi\log^{-(1+\xi)} n$ since
the area of a every sector is $\log^{1+\xi} n$, and the minimum circular radius
of a sector is $\tfrac{1}{\sqrt{2}}\log^{1+\xi} n$. Also, the angle between the
last ray $\cR_{I}$ and $\cR_1$ is smaller than 
$$\sum_{j=0}^{\infty}\frac{2\beta}{2^j}\leq 4\beta.$$ 
Since for $n$ large enough, the angle $\cR_0\cR_1$ is smaller than
$\tfrac{\pi}{2}+\tfrac{\pi}{12}$ we obtain $\alpha\ge 
\tfrac{\pi}{2}-\tfrac{\pi}{12}-\tfrac{\beta}{2}-4\beta$
which is positive for $n$ large enough.
\end{proof}

Let $M\subset\reals^2$ be defined by 
$$M:= b(x,\sqrt 2 \log^{1+\xi} n) \cup \bigcup_{i=1}^I Q_i,$$
that is the union of all shaded regions in Figure~\ref{fig:degree} (pink or blue). 

\begin{lemma} \label{lem:max_degree_points}    
    Suppose that every for every $0\le i\le I+1$, we have $|\ppp \cap S_i|>0$,
    and that the domain is rotated so that $\cR_0$ is in exactly in the
    direction of the $y$-axis. Then every neighbour of $x$ in $\DT(\ppp)$ having
    positive $x$-coordinate of must lie in $M$.
\end{lemma}
\begin{proof}
    Any neighbour $y$ of $x$ in $\DT(\ppp)$ must have a circle not containing
    any point of $\ppp$ that touches both $x$ and $y$. Let $y\notin M$. Since
    $y$ has positive $x$-coordinate, it is easy to see that any circle touching
    both $x$ and $y$ must also fully contain one of the sectors $S_i$, $0\le
    i\le I$, in pink in Figure~\ref{fig:degree} (see dotted circles in
    Figure~\ref{fig:degree}). By assumption, each $S_i$, $0\le i\le I$ contains
    at least a site of $\ppp$, so no circle touching both $x$ and $y$ can be
    empty, and $y$ cannot be a neighbour of $x$ in $\DT(\ppp)$.
\end{proof}
\begin{lemma}\label{lem:close_to_border}
    For any $\xi> 0$,
    $$
      \p\bigg(
        \max\bigg\{ 
          \delta_\ppp(x) 
          \;\big|\;
          x\in\ppp, \; \|\partial\cD_n\;x\| \leq \sqrt{\log n} 
        \bigg\}
        \geq 
        \log^{2+3\xi} n
      \bigg)
      \leq
      2\expo{-\log^{1+\xi} n},
    $$
    for $n$ large enough.        
\end{lemma}
\begin{proof}
  Let $X_0$ be chosen uniformly at random among the points of $\ppp$
  within distance $\sqrt{\log n}$ of $\partial \cD_n$. Note that if there is 
  no such point, then 
  $$\max \big\{ \delta_\ppp(x) : 
  x\in\ppp, \; \|\partial\cD_n\;x\| \leq \sqrt{\log n} \big\} =0.
  $$
  Let $A_{X_0}$ be the event that no (pink) sector $S_i$, $0\le i\le I$
  about $X_0$ is empty and let $B_{X_0}$ be the event that no sector
  $S_i$, $0\le i \le I$  about $X_0$ contains more than $\log^{1+2\xi} n$ points (see
  Figure~\ref{fig:degree}). Given that the number of sectors $I+2$ about $X_0$ is
  deterministically bounded by $\tfrac{\pi}{6}\log^{1+\xi}n\leq n$ (for $n$
  large enough), we have by the union bound that the probability that
  $$\p(A_{X_0}^c) \le n\cdot\exp\big(-\log^{1+\xi}\big)
  \qquad \text{and}\qquad
  \p(B_{X_0}^c) \le n\cdot\exp\big(-\log^{1+2\xi} n\big).$$
  Conditional on $A_{X_0}$ and $B_{X_0}$ occurring, we may count the
  number of points that could possibly be Delaunay neighbours of $X_0$. This
  includes all points in the at most $\tfrac{\pi}{6}\log^{1+\xi}n$ sectors, each
  containing at most $\log^{1+2\xi} n$ points (conditional on $B_{X_0}$). We
  also add all the points not contained within any sector, but lying within the
  circle of radius $\sqrt{2}\log^{1+\xi}$ about $X_0$ (shaded blue in
  Figure~\ref{fig:degree}). Standard arguments give that this region contains no
  more than $2\pi\log^{2+3\xi} n$ points with probability bounded at most
  $\expo{-\log^{2+3\xi}}$. Putting these together and applying the union bound
  we have
  \begin{align*}   
    & \p(\max \big\{ \delta_\ppp(x) : 
  x\in\ppp, \; \|\partial\cD_n\;x\| \leq \sqrt{\log n} \big\} > \log^{2+3\xi} n)\\
    & \le 2n\cdot\p\big(\delta_\ppp(X_0) 
    > 
    \log^{2+3\xi} \big) + \p\big(|\ppp| > 2n\big) \\
    & \leq
    2\exp\big(-\log^{1+\xi}\big),
  \end{align*}
  for $n$ sufficiently large.
\end{proof}
The second case of our proof is much simpler, since it suffices to bound the
maximum distance between any two Delaunay neighbours, and then count the maximum
number of points falling within this region.
\begin{lemma} \label{lem:degree_away_from_border}
    For any $\xi> 0$,
    $$
      \p\bigg(
        \max\bigg\{ 
          \delta_\ppp(x) 
          \;\big|\;
          x\in\ppp, \; \|\partial\cD_n\;x\| > \sqrt{\log n} 
        \bigg\}
        \geq 
        \log^{2+\xi} n
      \bigg)
      \leq
      \expo{-\log^{1+\xi/3} n},
    $$
    for $n$ large enough.    
\end{lemma}
\begin{proof}
  Let $X_0\in\ppp$ be chosen uniformly at random among the points of $\ppp$ 
  at distance more than $\sqrt{\log n}$ from $\partial \cD_n$. Again, note that 
  if there is no such point, 
  $$\max \big\{ \delta_\ppp(x) : 
  x\in\ppp, \; \|\partial\cD_n\;x\| > \sqrt{\log n} \big\} =0.
  $$
  If $X_0$ is well-defined, any neighbour of $X_0$ in $\DT(\ppp)$ outside of 
  the ball $b(X_0, 1/2 \log^{1/2+\xi} n)$
  implies the existence of large region of $\cD_n$ that is empty of points of $\ppp$. 
  By adapting the proof of Lemma~\ref{lem:bound_Yw} and using 
  Lemma~\ref{lem:largest_empty_disc}, the probability that such a neighbour
  exists may be bounded by $\expo{-\log^{1+\xi} n}$. We omit
  the details.

  We now upper bound the number of points that may fall within this ball. Split
  the domain into a regular grid with cells of side length 
  $\frac{1}{2}\log^{1/2+\xi} n$. The
  probability that any of these grid cells contains more than $\log^{2+\xi/3}$
  points is bounded by $\expo{-\log^{2+2\xi} n}$ for large $n$, and our ball may
  intersect at most four of these, so the degree of $X_0$ in $\DT(\ppp)$ is bounded by
  $4\log^{2+2\xi} n$ with probability $2\expo{-\log^{1+\xi} n}$ in this case.
  The result follows from the union bound, 
  just as in Lemma~\ref{lem:close_to_border}.
\end{proof}

\paragraph{Acknowledgments} We would like to thank 
Marc Glisse, Mordecai Golin, Jean-François Marckert, and Andrea Sportiello 
 for fruitful discussions during the Presage workshop 
on geometry and probability.

\bibliography{biblio}

\begin{thebibliography}{28}
\providecommand{\natexlab}[1]{#1}
\providecommand{\url}[1]{\texttt{#1}}
\expandafter\ifx\csname urlstyle\endcsname\relax
  \providecommand{\doi}[1]{doi: #1}\else
  \providecommand{\doi}{doi: \begingroup \urlstyle{rm}\Url}\fi

\bibitem[Baccelli and Blaszczyszyn(2009)]{bb-sgwn-09}
F.~Baccelli and B.~Blaszczyszyn.
\newblock \emph{Stochastic Geometry and Wireless Networks}.
\newblock NOW, 2009.

\bibitem[Bern et~al.(1991)Bern, Eppstein, and Yao]{BernEY91}
M.~W. Bern, D.~Eppstein, and F.~F. Yao.
\newblock The expected extremes in a delaunay triangulation.
\newblock In \emph{ICALP}, pages 674--685, 1991.

\bibitem[Bonichon and Marckert(2011)]{BoMa2011a}
N.~Bonichon and J.-F. Marckert.
\newblock {Asymptotics of geometrical navigation on a random set of points in
  the plane}.
\newblock \emph{Adv. in Applied Probability}, 43:\penalty0 899--942, 2011.

\bibitem[Bordenave(2008)]{Bordenave2008}
C.~Bordenave.
\newblock Navigation on a {P}oisson point process.
\newblock \emph{The Annals of Applied Probability}, 18:\penalty0 708--746,
  2008.

\bibitem[Bose and Devroye(2006)]{BoDe2006}
P.~Bose and L.~Devroye.
\newblock On the stabbing number of a random {D}elaunay triangulation.
\newblock \emph{Computational Geometry: Theory and Applications}, 36:\penalty0
  89--105, 2006.

\bibitem[Bose and Morin(2004{\natexlab{a}})]{BoMo2004a}
P.~Bose and P.~Morin.
\newblock Online routing in triangulations.
\newblock \emph{SIAM journal on computing}, 33:\penalty0 937--951,
  2004{\natexlab{a}}.

\bibitem[Bose and Morin(2004{\natexlab{b}})]{DBLP:journals/tcs/BoseM04}
P.~Bose and P.~Morin.
\newblock Competitive online routing in geometric graphs.
\newblock \emph{Theor. Comput. Sci.}, 324\penalty0 (2-3):\penalty0 273--288,
  2004{\natexlab{b}}.

\bibitem[Bose et~al.(2002)Bose, Brodnik, Carlsson, Demaine, Fleischer,
  L{\'o}pez-Ortiz, Morin, and Munro]{BoBrCaDe2002a}
P.~Bose, A.~Brodnik, S.~Carlsson, E.~Demaine, R.~Fleischer, A.~L{\'o}pez-Ortiz,
  P.~Morin, and J.~Munro.
\newblock {Online routing in convex subdivisions}.
\newblock \emph{International Journal of Computational Geometry \&
  Applications}, 12:\penalty0 283--295, 2002.

\bibitem[Boucheron et~al.(2012)Boucheron, Lugosi, and Massart]{BoLuMa2012a}
S.~Boucheron, G.~Lugosi, and P.~Massart.
\newblock \emph{{Concentration Inequalities - A nonasymptotic theory of
  independence}}.
\newblock Clarendon Press, 2012.

\bibitem[Chen et~al.(2012)Chen, Devroye, Dujmovi{\'c}, and Morin]{Chen2012178}
D.~Chen, L.~Devroye, V.~Dujmovi{\'c}, and P.~Morin.
\newblock Memoryless routing in convex subdivisions: Random walks are optimal.
\newblock \emph{Computational Geometry}, 45\penalty0 (4):\penalty0 178 -- 185,
  2012.

\bibitem[{de Castro} and Devillers(2011)]{geometrica-7322i}
P.~M.~M. {de Castro} and O.~Devillers.
\newblock Simple and efficient distribution-sensitive point location, in
  triangulations.
\newblock In \emph{ALENEX}, pages 127--138, 2011.

\bibitem[Devillers et~al.(2002)Devillers, Pion, and Teillaud]{dpt-wt-02}
O.~Devillers, S.~Pion, and M.~Teillaud.
\newblock Walking in a triangulation.
\newblock \emph{Int. J. Found. Comput. Sci.}, 13:\penalty0 181--199, 2002.

\bibitem[Devroye et~al.(1998)Devroye, M{\"u}cke, and Zhu]{DeMuZh1998a}
L.~Devroye, E.~M{\"u}cke, and B.~Zhu.
\newblock {A note on point location in Delaunay triangulations of random
  points}.
\newblock \emph{Algorithmica}, 22:\penalty0 477--482, 1998.

\bibitem[Devroye et~al.(2004)Devroye, Lemaire, and Moreau]{dlm-etadp-04}
L.~Devroye, C.~Lemaire, and J.-M. Moreau.
\newblock Expected time analysis for {Delaunay} point location.
\newblock \emph{Comput. Geom. Theory Appl.}, 29:\penalty0 61--89, 2004.

\bibitem[Dubhashi and Panconesi(2009)]{DuPa2009}
D.~Dubhashi and A.~Panconesi.
\newblock \emph{{Concentration of Measure for the Analysis of Randomized
  Algorithms}}.
\newblock Cambridge University Press, 2009.

\bibitem[Janson(2004)]{Janson2004b}
S.~Janson.
\newblock Large deviation for sums of partially dependent random variables.
\newblock \emph{Random Structures and Algorithms}, 24\penalty0 (3):\penalty0
  234--248, 2004.

\bibitem[Kozma et~al.(2004)Kozma, Lotker, and
  Sharir]{Kozma04geometricallyaware}
G.~Kozma, Z.~Lotker, and M.~Sharir.
\newblock Geometrically aware communication in random wireless networks.
\newblock \emph{ACM PODC}, 2004.

\bibitem[Lawson(1977)]{l-scsi-77}
C.~L. Lawson.
\newblock Software for {$C^{1}$} surface interpolation.
\newblock In J.~R. Rice, editor, \emph{Math. Software III}, pages 161--194.
  Academic Press, New York, NY, 1977.

\bibitem[Li et~al.(2003)Li, Calinescu, Wan, and
  Wang]{Li:2003:LDT:2328757.2328820}
X.-Y. Li, G.~Calinescu, P.-J. Wan, and Y.~Wang.
\newblock Localized delaunay triangulation with application in ad hoc wireless
  networks.
\newblock \emph{IEEE Trans. Parallel Distrib. Syst.}, 14\penalty0
  (10):\penalty0 1035--1047, 2003.

\bibitem[Matou\v{s}ek(2002)]{Matousek2002a}
J.~Matou\v{s}ek.
\newblock \emph{{Lectures on discrete geometry}}.
\newblock Springer, 2002.

\bibitem[McDiarmid(1998)]{McDiarmid1998}
C.~J.~H. McDiarmid.
\newblock Concentration.
\newblock In {Habib et al.}, editor, \emph{Probabilistic Methods in Algorithmic
  Discrete Mathematics}, pages 195--248. Springer-Verlag, 1998.

\bibitem[Misra et~al.(2009)Misra, Woungang, and Misra]{MiWoMi2009a}
S.~Misra, I.~Woungang, and S.~Misra.
\newblock \emph{{Guide to wireless sensor networks}}.
\newblock Springer, 2009.

\bibitem[M{\"u}cke et~al.(1996)M{\"u}cke, Saias, and Zhu]{Mucke:1996aa}
E.~P. M{\"u}cke, I.~Saias, and B.~Zhu.
\newblock Fast randomized point location without preprocessing in two- and
  three-dimensional delaunay triangulations.
\newblock \emph{IN PROC. 12TH ANNU. ACM SYMPOS. COMPUT. GEOM}, 12:\penalty0
  274--283, 1996.

\bibitem[Wang and Li(2007)]{Wang:2007:EDL:1269960.1269963}
Y.~Wang and X.-Y. Li.
\newblock Efficient delaunay-based localized routing for wireless sensor
  networks: Research articles.
\newblock \emph{Int. J. Commun. Syst.}, 20\penalty0 (7):\penalty0 767--789,
  2007.

\bibitem[Xia(2011)]{Xia2011a}
G.~Xia.
\newblock {Improved upper bound on the stretch factor of Delaunay
  triangulations}.
\newblock In \emph{Proceedings of the 27th annual ACM symposium on
  Computational geometry (SoCG)}, pages 264--273. ACM, 2011.

\bibitem[Xia and Zhang(2011)]{DBLP:conf/cccg/XiaZ11}
G.~Xia and L.~Zhang.
\newblock Toward the tight bound of the stretch factor of delaunay
  triangulations.
\newblock In \emph{CCCG}, 2011.

\bibitem[Zhao and Guibas(2004)]{zg-wsnip-04}
F.~Zhao and L.~Guibas.
\newblock \emph{Wireless Sensor Networks. An Information Processing Approach}.
\newblock Morgan Kaufmann, 2004.

\bibitem[Zhu(2003)]{z-lowrd-03}
B.~Zhu.
\newblock On {Lawson's} oriented walk in random {Delaunay} triangulations.
\newblock In \emph{Fundamentals of Computation Theory}, volume 2751 of
  \emph{Lecture Notes Comput. Sci.}, pages 222--233. Springer-Verlag, 2003.

\end{thebibliography}
\bibliographystyle{abbrvnaturl}

\end{document}